\documentclass[11pt]{article}
\usepackage{natbib}
\usepackage{graphicx}%
\usepackage{multirow}%
\usepackage{amsmath,amssymb,amsfonts}%
\usepackage[bb=boondox]{mathalpha}
\usepackage{amsthm}%
\usepackage{mathrsfs}%
\usepackage{xcolor}%
\usepackage{textcomp}%
\usepackage{manyfoot}%
\usepackage{booktabs}%
\usepackage{placeins}
\usepackage{algorithm}%
\usepackage{algorithmicx}%
\usepackage{algpseudocode}%
\usepackage{listings}%
\usepackage{a4wide}

\usepackage{bm}
\usepackage{bbm}
\usepackage{subcaption}
\usepackage{algorithm}
\usepackage{algpseudocode}

\newcommand{\E}{\mathbb E}
\newcommand{\e}{\mathrm{e}}
\newcommand{\Var}{\mathrm{Var}}

\newcommand{\D}{\mathrm{d}}

\theoremstyle{definition}
\newtheorem{theorem}{Theorem}
\newtheorem{lemma}[theorem]{Lemma}
\newtheorem{corollary}[theorem]{Corollary}
\newtheorem{proposition}[theorem]{Proposition}

\newtheorem{assumption}{Assumption}
\newtheorem{definition}{Definition}

\begin{document}
\title{A Markov duration model with power-law decay \mbox{via level--slope matching}}

\author{Kyungsub Lee}

\maketitle

\begin{abstract}
We develop a two-dimensional Markov duration model with power-law
rate decay and a flexible residual distribution. The two state
variables specify the excitation level and decay scale of a single
power-law rate curve, and each update matches the level and slope
jumps of a power-law Hawkes intensity. Under a sufficient stability
condition, a finite residual mean, and a positive continuous residual
density, the state at event times is positive Harris recurrent and
converges geometrically in total variation to a unique invariant
distribution. With unit exponential residuals, the sufficient region
contains every parameter value for which the matched Hawkes process
is subcritical.
We apply the model to E-mini S\&P 500 futures and NVDA midquote
durations at several price-change thresholds. Flexible residual laws
improve likelihood-based fit, and power-law decay lowers one-step
negative log predictive scores relative to exponential decay. An
autoregressive conditional duration model generally leaves less
residual autocorrelation, while forecast rankings vary with the
price threshold and evaluation criterion.
\end{abstract}

\section{Introduction}

A basic forecasting question in high-frequency finance is how soon
the next trade or midquote change will occur. A duration model
translates recent market activity into a distribution for the next
duration. The Markov power-law model carries forward two state
variables after each event. One specifies how much activity is
elevated above a baseline, and the other controls how slowly it
subsides. The next observed duration updates both, so the same
recursion can be used to forecast another duration or evaluate the
likelihood.

In a self-exciting point process, each event raises the rate of
further events. The Hawkes process \citep{Hawkes1, Hawkes2} is the
canonical example: the excitation from each event decays over time
according to a prescribed kernel. Power-law Hawkes kernels describe
slowly decaying excitation and have been used to model high-frequency
financial events \citep{hardiman2013critical, bacry2016estimation,
zhang2016modeling, nystrom2022hawkes}.

With an exponential kernel, contributions from different events share
a common decay factor, so the current intensity suffices to propagate
the rate up to the next event. For a power-law kernel, the factor
depends on the age of each event, so the current intensity alone does
not determine its subsequent decay. This complicates likelihood-based
inference and simulation. Existing approaches approximate the kernel
by sums of exponentials \citep{khabou2025markov} or estimate it
nonparametrically while retaining the full event history
\citep{bacry2011nonparametric, zhang2019efficient,
Zhuang2020NonparametricFramework, bonnet2025nonparametric}.

Power-law decay between events still leaves a choice about what to
carry forward at the next event. An update that carries forward only
the level and starts a new curve with a fixed decay scale gives any
surviving excitation the same decay scale as the new shock. We match
the slope as well, so that the surviving excitation determines how
slowly the new curve decays. Both pieces of information fit in a
two-dimensional Markov state.

The model also separates the rate curve from the residual
distribution. Following the flexible-residual construction of
\citet{lee2025self}, each duration is obtained by inverting the
integrated rate at the current state at a fresh iid residual. In
duration econometrics, the autoregressive conditional duration (ACD)
model of \citet{EngleRussell1998} rescales an iid residual by the
conditional mean duration, and mixtures of exponentials and other
flexible laws improve its fit and forecasts
\citep{deluca2004mixture,yatigammana2019forecasting}.
\citet{lee2026forecasting} studies stability for an exponentially
decaying rate within the flexible-residual framework and compares
residual distributions for forecasting financial durations. Here we
retain the flexible residual distribution and establish stability for
a power-law rate with a two-dimensional state.

Our main theoretical result is a sufficient condition for stability,
\[
 \xi\gamma<\frac{\mu_\varepsilon p^2}{p+1},
\]
where $\xi$ is the jump in level at an event, $\gamma$ the decay scale
of a new contribution, $p$ the decay exponent, and $\mu_\varepsilon$
the residual mean. Under this condition, a finite residual mean, and a
positive continuous residual density, the state at event times is
positive Harris recurrent and converges geometrically in total
variation to a unique invariant distribution from every initial state
(Theorem~\ref{Thm:FL_Joint}). With unit exponential residuals, the
condition holds at every parameter value for which the matched Hawkes
process is subcritical. The proof combines a drift bound that controls
unusually high activity or slow decay with a minorization: a
sufficiently long duration brings the state close to the contribution
of a single new event, from which different starting points reach
overlapping distributions.

We use E-mini S\&P 500 futures and NVDA midquote durations to examine
how rate decay and residual flexibility affect fit and one-step
forecasting. A two-component exponential mixture improves fit and
log predictive scores under both power-law and exponential rate
dynamics. Under this common residual law, the proposed model has
better log predictive scores than the self-exciting exponential-decay
benchmark (SE), while comparisons with ACD depend on the price
threshold. ACD has the lowest RMSE and generally leaves less residual
autocorrelation. Density scores, quantile losses, and tail exceedance
frequencies favor different specifications.

Other Markovian or memory-based constructions include the
discrete-time Hawkes process with inhibition and finite memory of
\citet{CostaMaillardMuraro2024}, the renewal-based stability analysis
of signed Hawkes processes in \citet{costa2020renewal}, the
ephemerally self-exciting process of \citet{daw2022ephemerally}, which
confines each event's excitation to a random lifetime, variable-length
memory formulations \citep{hodara2017hawkes,
quayle2025hawkesprocessesvariablelength}, and the mean-field Hawkes
process with age and leaky memory of \citet{Schmutz2022MeanField}. Our
model differs in its update, which matches the level and slope of a
power-law rate in a fixed two-dimensional state. The stability
argument follows the drift and minorization approach used for
Markovian state-space models
\citep{MayerhoferStelzerVestweber2020AffineCones,ShaoWangWu2023HybridSwitching}
and for Markovian representations of Hawkes intensities
\citep{PalmowskiPojerThonhauser2025RuinHawkes,KwanChenDunsmuir2025ErgodicGeneralKernelHawkes}.

The remainder of the paper is organized as follows.
Section~\ref{Sect:model} introduces the model: the flexible-residual construction, the power-law Hawkes jumps that motivate the rate, and the level and slope matching update.
Section~\ref{Sect:MC} discusses the Markov properties of the model.
Section~\ref{Sect:empirical} studies midquote durations for E-mini futures and NVDA across three price thresholds.
Section~\ref{Sect:conc} concludes.

\section{Model}\label{Sect:model}

The Markov power-law model specifies a point process on
$[0,\infty)$ through its
durations: after each event, a state determines the latent rate until
the next event.

\subsection{From the current state to the next duration}

Let $T_n$ be the time of event $n$, set $T_0=0$, and write
$\tau_n=T_n-T_{n-1}$ for the duration (inter-arrival time). The model uses two operations.
Given the state after the previous event, draw a residual (or
innovation) and wait
until the accumulated latent rate reaches that draw. At the new
event, use the realized duration to update the state. This is the
flexible-residual construction of \citet{lee2025self}. With unit
exponential residuals, the latent rate is the conditional intensity
of the point process.

The next definition states this construction for a general state. In
it, $\Phi(t,s)$ is the accumulated latent rate at time $t$ after an
event with state $s$, and $\Psi$ is the update rule: $\Psi(s,t)$ is
the new state when the next event arrives after duration $t$. Section~\ref{Sect:matching}
specifies both for our model.

\begin{definition}\label{Def:general_vector}
Let $S_0$ take values in a Borel state space
$\mathcal S\subseteq\mathbb R^d$. Independently of $S_0$, draw iid
positive residuals $(\varepsilon_n)_{n\ge1}$ with common cumulative
distribution function $F_\varepsilon$. Let
$\Phi:[0,\infty)\times\mathcal S\to[0,\infty)$ be jointly measurable,
with $\Phi(\cdot,s)$ a continuous strictly increasing bijection
for each $s$. Let $\Psi:\mathcal S\times[0,\infty)\to\mathcal S$ be
measurable. For $n\ge1$, set
\begin{align}
 \tau_n &= \Phi^{-1}(\varepsilon_n,S_{n-1}), \label{Eq:tau_vector}\\
 S_n &= \Psi(S_{n-1},\tau_n). \label{Eq:Psi0}
\end{align}
The inverse is taken in the time argument of $\Phi$. If
$T_n=\sum_{i=1}^n\tau_i$ tends to infinity almost surely, the event
times define the point process $N=\sum_{n\ge1}\delta_{T_n}$, where
$\delta_t$ denotes the Dirac measure at $t$.
\end{definition}

Given $S_{n-1}=s$, the probability that the next duration exceeds $t$ is
$1-F_\varepsilon(\Phi(t,s))$. For unit exponential residuals, this
is $\exp(-\Phi(t,s))$.

\subsection{Power-law Hawkes jumps}\label{Sect:motive}

Consider a linear Hawkes process with event times $(T_i^{\mathrm H})$
and conditional intensity
\[
\lambda_{\mathrm H}(t)=\mu+\sum_{T_i^{\mathrm H}<t} g(t-T_i^{\mathrm H}),
\qquad 
g(u)=\frac{\alpha}{(u+\gamma)^p}, \quad u>0,
\]
where $\mu,\alpha,\gamma>0$ and $p>1$. For large $u$, $g(u)$ is close
to $\alpha u^{-p}$, so the effect of an event decays as a power of the
elapsed time. Since $p>1$, the kernel is integrable.

At an event time $T_n^{\mathrm H}$, the new event changes the level
and the slope of the intensity by
\begin{align}
	\lambda_{\mathrm H}((T_n^{\mathrm H})^+)-\lambda_{\mathrm H}((T_n^{\mathrm H})^-)
	&= g(0)=\frac{\alpha}{\gamma^p}, \label{Eq:update1}\\
	\lambda_{\mathrm H}'((T_n^{\mathrm H})^+)-\lambda_{\mathrm H}'((T_n^{\mathrm H})^-)
	&= g'(0)=-\frac{\alpha p}{\gamma^{p+1}}. \label{Eq:update2}
\end{align}
With an exponential kernel, the current level of the intensity
determines its future path, so the intensity is a one-dimensional
Markov process. With the power-law kernel, the intensity between
events is a sum of curves of different ages, and its future path
depends on all past event times.

A single event contributes the curve
$g(t)=\xi(1+t/\gamma)^{-p}$ with $\xi=\alpha/\gamma^p$, which is
determined by its level $\xi$ and its decay scale $\gamma$. Our model
keeps one curve of this shape at all times. After each event, it
replaces the current curve by a new one whose level and slope change
by the Hawkes jumps \eqref{Eq:update1}--\eqref{Eq:update2}. The
next subsection gives this update.

\subsection{Level and slope matching}\label{Sect:matching}

The state after event $n$ is $S_n=(X_n,c_n)$, where $X_n$ is the
excitation level, the part of the latent rate above $\mu$, and $c_n$
is the decay scale.
With $t$ denoting time since that event, the latent rate is
$\phi_n(t)=\phi(t;X_n,c_n)$, where
\begin{equation}\label{eq:intensity2}
 \phi(t;x,c)=\mu+x\left(1+\frac{t}{c}\right)^{-p}.
\end{equation}
Its initial level is $\mu+x$ and its initial slope is $-px/c$.
A larger $c$ makes excitation decay more slowly at a given $x$.

Recall that $\xi=\alpha/\gamma^p$ is the jump in level from
\eqref{Eq:update1}. When the next event arrives after a duration
$\tau_n$, the previous curve has already decayed to age $\tau_n$.
To incorporate the event, we choose the level and slope of the
new curve to satisfy
\[
 \phi_n(0)=\phi_{n-1}(\tau_n)+\xi,\qquad
 \phi_n'(0)=\phi_{n-1}'(\tau_n)-p\xi/\gamma .
\]
Thus $\gamma$ is the decay scale of the new excitation. In the
coordinates $(X_n,X_n/c_n)$, these two equations simply add
$(\xi,\xi/\gamma)$ to the corresponding decayed coordinates.

\begin{definition}\label{Def:Model}
Use the recursion in Definition~\ref{Def:general_vector} with
$\mathcal S=[\xi,\infty)\times[\gamma,\infty)$ and parameters
$\mu,\xi,\gamma>0$ and $p>1$. Take $S_0$ in $\mathcal S$.
For a state $s=(x,c)$ and duration $t$, write
$\Psi(s,t)=(\Psi_x(x,c,t),\Psi_c(x,c,t))$, where
\begin{align}
 \Psi_x(x,c,t) &= x\left(1+\frac{t}{c}\right)^{-p}+\xi, \label{Eq:Psi_x}\\
 \Psi_c(x,c,t) &=
 \frac{\Psi_x(x,c,t)}{(x/c)(1+t/c)^{-(p+1)}+\xi/\gamma}.
 \label{Eq:Psi_c}
\end{align}
These are the level and slope matching equations solved for the
two new state coordinates. The accumulated latent rate is
\begin{equation}\label{Eq:Phi}
 \Phi(t,x,c)=\int_0^t\phi(u;x,c)\,\D u
 =\mu t+\frac{xc}{p-1}\left[1-\left(1+\frac{t}{c}\right)^{1-p}\right].
\end{equation}
\end{definition}

Slope matching carries the decay scale of the surviving excitation
into the next state. From state $(x,c)$, after a duration $t$, that
excitation has level $a=x(1+t/c)^{-p}$ and, seen from the new event
time, decays as $a(1+u/(c+t))^{-p}$, that is, with scale $c+t$. The
new event contributes level $\xi$ with scale $\gamma$. By
\eqref{Eq:Psi_c}, the updated scale $c^+=\Psi_c(x,c,t)$ satisfies
\begin{equation}\label{Eq:scale_carryover}
 \frac{1}{c^+}
 =\frac{a}{a+\xi}\frac{1}{c+t}
  +\frac{\xi}{a+\xi}\frac{1}{\gamma}.
\end{equation}
Hence $\gamma\le c^+\le c+t$. When $t>0$, $c^+>\gamma$, so the
surviving excitation increases the decay scale above that of a new
contribution.

We use $\varepsilon_n$ for the random residual and $\epsilon$ for
a fixed residual value. The recursion describes a valid event process
if every draw yields
a unique finite duration, the updated state remains admissible, and
infinitely many events cannot accumulate in finite time. The next
proposition verifies these properties. Continuity of the update
will also be used to study which states the chain can reach.

\begin{proposition}[Well-posedness and non-explosion]\label{Prop:well_posedness}
Under the assumptions of Definition~\ref{Def:Model}, the following hold.
\begin{enumerate}
 \item[(i)] For every $(x,c)\in\mathcal S$ and $\epsilon\ge0$, the equation
 $\Phi(t,x,c)=\epsilon$ has a unique solution $t\ge0$. The updated state
 belongs to $\mathcal S$, and both the duration and the updated state
 depend jointly continuously on $(\epsilon,x,c)$.
 \item[(ii)] For every initial state $S_0\in\mathcal S$, the event times
 satisfy $T_n\to\infty$ almost surely. In particular, the point process has
 finitely many events in every bounded time interval almost surely.
\end{enumerate}
\end{proposition}

\begin{proof}
\textup{(i)} Fix $(x,c)\in\mathcal S$. The map $\Phi(\cdot,x,c)$ is continuous,
$\Phi(0,x,c)=0$, and $\partial_t\Phi(t,x,c)=\mu+x(1+t/c)^{-p}\ge\mu>0$.
Hence $\Phi(\cdot,x,c)$ is strictly increasing with $\Phi(t,x,c)\ge\mu t$,
and it takes each value $\epsilon\ge0$ exactly once.
For joint continuity, let $(\epsilon_j,x_j,c_j)\to(\epsilon,x,c)$ and let
$t_j$ be the corresponding durations. Since $0\le t_j\le\epsilon_j/\mu$,
the sequence is bounded. By continuity of $\Phi$, every subsequential
limit solves $\Phi(t,x,c)=\epsilon$, so uniqueness gives $t_j\to t$.

Write $t$ for this duration and $a=x(1+t/c)^{-p}$. The updated level
is $a+\xi\ge\xi$. By \eqref{Eq:scale_carryover}, the updated scale
lies in $[\gamma,c+t]$. The updated state therefore belongs to
$\mathcal S$, and both coordinates are continuous in $(\epsilon,x,c)$.

\textup{(ii)} Fix $S_0=(X_0,c_0)\in\mathcal S$. It suffices to show that
$\sum_n\tau_n=\infty$ almost surely. Since
$X_n=X_{n-1}(1+\tau_n/c_{n-1})^{-p}+\xi\le X_{n-1}+\xi$, we have
$X_{n-1}\le X_0+(n-1)\xi$. Since $\partial_t\Phi(t,x,c)\le\mu+x$,
\[
 \varepsilon_n=\Phi(\tau_n,X_{n-1},c_{n-1})\le(\mu+X_{n-1})\tau_n,
 \qquad\text{so}\qquad
 \tau_n\ge\frac{\varepsilon_n}{w_n},\quad w_n=\mu+X_0+(n-1)\xi .
\]
To avoid moment assumptions on the residuals, truncate them: let
$Y_n=\min(\varepsilon_n,1)/w_n\le\tau_n$. The $Y_n$ are independent,
$\sum_n\E Y_n=\E[\min(\varepsilon_1,1)]\sum_n w_n^{-1}=\infty$, and
$\sum_n\Var(Y_n)\le\sum_n w_n^{-2}<\infty$. By Kolmogorov's convergence
theorem, $\sum_n(Y_n-\E Y_n)$ converges almost surely. Hence
$\sum_n\tau_n\ge\sum_nY_n=\sum_n(Y_n-\E Y_n)+\sum_n\E Y_n=\infty$
almost surely.
\end{proof}

The stability condition uses the mean residual
$\mu_\varepsilon=\E\varepsilon_1$. A separate assumption on the
shape of its distribution lets us use changes of variables to obtain
transition densities and the conditional density of an observed duration.
\begin{assumption}[Residual density]\label{Ass:residual_density}
The residual distribution has a density $f_\varepsilon$ that is continuous
and strictly positive on $(0,\infty)$.
\end{assumption}

Let $\E_{x,c}$ denote expectation conditional on $S_{n-1}=(x,c)$.
The conditional survival function is
$\mathbb{P}(\tau_n>t\mid S_{n-1}=(x,c))=\bar F_\varepsilon(\Phi(t,x,c))$,
where $\bar F_\varepsilon=1-F_\varepsilon$. Under
Assumption~\ref{Ass:residual_density}, the conditional hazard of the
duration is the point-process intensity between events relative to the
filtration generated by $S_0$ and the observed event history. The
initial state and past durations determine $S_{n-1}=(x,c)$, so
\begin{equation}\label{Eq:flexible_hazard}
 \lambda_{n-1}(t;x,c)
 =h_\varepsilon(\Phi(t,x,c))\,\phi(t;x,c),
 \qquad h_\varepsilon(u)=\frac{f_\varepsilon(u)}{\bar F_\varepsilon(u)}.
\end{equation}
This is the hazard relation used by \citet{lee2026forecasting}, with
the power-law latent rate here in place of the exponentially decaying
rate. For unit exponential residuals $h_\varepsilon=1$, so the latent
rate is the conditional intensity. For other residual distributions the hazard
factor changes the conditional duration distribution while the
latent rate still follows the same update equations. Gamma and
generalized gamma distributions, and Burr distributions with a finite mean, are examples
that meet the density and moment requirements; the recurrence
theorem also requires the parameter inequality \eqref{Eq:condition}.

For $k>0$, Tonelli's theorem gives
\begin{equation}\label{Eq:CM_tau}
 \E_{x,c}[\tau_n^k]
 =\int_0^\infty kt^{k-1}\bar F_\varepsilon(\Phi(t,x,c))\,\D t .
\end{equation}
In particular,
\begin{equation}\label{Eq:CE_tau}
 \E_{x,c}[\tau_n]=\int_0^\infty\bar F_\varepsilon(\Phi(t,x,c))\,\D t
 \le \mu_\varepsilon/\mu
\end{equation}
when $\mu_\varepsilon<\infty$, since $\tau_n\le\varepsilon_n/\mu$.

When the residual distribution has a density, the conditional density of $\tau_n$
given $S_{n-1}=(x,c)$ is
\[
 f_{\tau\mid x,c}(t)=
 f_\varepsilon(\Phi(t,x,c))\,\partial_t\Phi(t,x,c),\qquad t>0.
\]
The log-likelihood of the observed durations given $S_0$ is
\begin{equation}
 \log L=\sum_{n=1}^m\left[
 \log f_\varepsilon\!\left(\Phi(\tau_n,X_{n-1},c_{n-1})\right)
 +\log\phi(\tau_n;X_{n-1},c_{n-1})
 \right].
\end{equation}
Here $m$ is the number of observed durations. Each likelihood term
uses the state computed from the preceding durations and the
initial state $S_0$.

The scale of the residual is not identified separately from the rate.
If $\mu$, $\xi$, and $X_0$ are multiplied by a constant $k>0$, then
$c_n$ is unchanged and $\Phi$ is multiplied by $k$, so replacing
$\varepsilon_n$ by $k\varepsilon_n$ leaves every duration unchanged.
When the residual distribution has a free scale parameter, we therefore
fix its scale, for example by $\E\varepsilon_1=1$, as for unit
exponential residuals.

\section{Markov chain}\label{Sect:MC}

We first ask what brings the state back from an unusually large
excitation level or slow decay. We answer this by finding a function of the state
that decreases on average outside a bounded region. We then show
that a long duration can bring different starting states close to
$s_*=(\xi,\gamma)$, the state associated with a single new event.
From there, two further steps provide a common part of their
transition distributions. These two results give the convergence theorem.

For $s=(x,c)\in\mathcal S$ and a residual value $\epsilon\ge0$, write
$t_s(\epsilon)=\Phi^{-1}(\epsilon,x,c)$. Since the residuals are iid
and the same update applies at every event, we can describe the
transition from $S_0=s$ to $S_1$. For a Borel set $B\subseteq\mathcal S$,
\begin{equation}
 P(s,B)=\mathbb{P}\bigl(\Psi(s,t_s(\varepsilon_1))\in B\bigr).
 \label{Eq:P}
\end{equation}
The $m$-step kernel $P^m(s,B)$ is the probability of reaching $B$
after $m$ more events when the current state is $s$. Write
$\mathcal B(\mathcal S)$ for the Borel sets of $\mathcal S$.
Appendix~\ref{App:MC} collects the Markov chain terms used below and
the convergence theorem we apply.

\subsection{Drift at large excitation level and slow decay}\label{Sect:joint_drift}

To control states far from $s_*$, we measure their size by
$V_q(x,c)=xc^q$ with $0<q<1$. Since $x\ge\xi$ and $c\ge\gamma$, we
have $V_q\ge\max(\gamma^q x,\ \xi c^q)$, so $V_q$ is large when either
$x$ or $c$ is large, and each set
$\{V_q\le M\}$ is compact; that is, $V_q$ is norm-like in the sense of
\citet{meyn2009markov}. The Lyapunov function in the drift condition
is $W=\exp(\theta V_q)$ for a small $\theta>0$. We first bound the
one-step increase of
$V_q$ for every duration. We then compute its limiting change as
$x\to\infty$, where durations are short.

Write $PV(s)=\E[V(\Psi(s,t_s(\varepsilon_1)))]$ for the one-step
expectation of a nonnegative function $V$, and define the increment
\[
 \Delta_q(s,\epsilon)=V_q(\Psi(s,t_s(\epsilon)))-V_q(s).
\]

The first bound limits how much $V_q$ can increase at a single
event. For any $\theta>0$, it gives
$\exp(\theta\Delta_q)\le\exp(\theta\xi\gamma^q)$, a uniform bound
we use when taking expectations in the geometric drift proof.
The second bound is negative once $c$ is large enough, because
$-\xi c^q$ outweighs the constant term. Then the update decreases
$V_q$ for every duration, even an arbitrarily short one.

\begin{lemma}[Bounds on the increment]\label{Lem:fractional_bounds}
For every $s=(x,c)\in\mathcal S$ and $\epsilon\ge0$,
\[
 \Delta_q(s,\epsilon)\le \xi\gamma^q,\qquad
 \Delta_q(s,\epsilon)\le 2\xi(2\gamma)^q-\xi c^q.
\]
\end{lemma}
\begin{proof}
Write $V_q(x,c)=Q(x,x/c)$, where
$Q(u,v)=u^{1+q}v^{-q}$ for $u\ge0$, $v>0$, and $Q(0,0)=0$.
For a vector $\mathbf a=(a_1,a_2)$, write $Q(\mathbf a)=Q(a_1,a_2)$.
As the perspective of the convex function $z\mapsto z^{1+q}$,
$Q$ is convex and positively homogeneous. For vectors $\mathbf a$
and $\mathbf b$ in its domain, these properties give
\[
 Q(\mathbf a+\mathbf b)
 =2Q\bigl((\mathbf a+\mathbf b)/2\bigr)
 \le Q(\mathbf a)+Q(\mathbf b).
\]
Set $t=t_s(\epsilon)$, $(x_t,c_t)=\Psi(s,t)$, and $r=1+t/c$.
At the event, the updated level and slope add the remaining
contribution of the old curve to the contribution of the new event.
The second coordinate below is the slope divided by $-p$.
\[
 \left(x_t,\frac{x_t}{c_t}\right)
 =\underbrace{\left(xr^{-p},\frac{xr^{-p}}{c+t}\right)}_{\text{remaining excitation}}
 +\underbrace{\left(\xi,\frac{\xi}{\gamma}\right)}_{\text{new contribution}}.
\]
Subadditivity gives $V_q(\Psi(s,t))\le xc^q r^{q-p}+\xi\gamma^q$,
which yields the first bound since $r\ge1$ and $p>q$.

For the second bound, we want to show how a new event reduces
the contribution of a large decay scale $c$ to $V_q$. Before
decay, split the old excitation level $x$ into $x-\xi$ and $\xi$,
which is possible since $x\ge\xi$. Both pieces decay along the
old curve. At the event, group the aged $\xi$ piece with the
new contribution, leaving the other piece on its own.
\[
 \left(x_t,\frac{x_t}{c_t}\right)
 =\underbrace{\left((x-\xi)r^{-p},\,\frac{(x-\xi)r^{-p}}{c+t}\right)}
              _{\text{aged remainder}}
 +\underbrace{\left(\xi r^{-p}+\xi,\,\frac{\xi r^{-p}}{c+t}+\frac{\xi}{\gamma}\right)}
              _{\text{aged piece plus new contribution}}.
\]
The aged remainder contributes $(x-\xi)c^q r^{q-p}\le(x-\xi)c^q$.
In the second vector $(u,v)$, the aged piece has level
$\xi r^{-p}\le\xi$ and the new event adds another $\xi$, so
$u\le2\xi$. The new event alone contributes $\xi/\gamma$ to $v$.
Thus $u/v\le 2\xi/(\xi/\gamma)=2\gamma$ and
$Q(u,v)\le2\xi(2\gamma)^q$. Adding the two bounds by subadditivity
and subtracting the old value gives
\[
 \Delta_q(s,\epsilon)
 \le (x-\xi)c^q+2\xi(2\gamma)^q-xc^q
 =2\xi(2\gamma)^q-\xi c^q.
\]
\end{proof}

For a fixed residual and bounded $c$, the duration tends to
zero as $x$ increases. Lemma~\ref{Lem:fractional_limit} shows that
the one-step change in $V_q$ converges to a limit that decreases
linearly with the residual. We use the negative mean of this
limit to establish geometric drift for large $x$.

\begin{lemma}[Large-level limit]\label{Lem:fractional_limit}
Let $(x_j,c_j)_{j\ge1}$ be a sequence of starting states in
$\mathcal S$ with $x_j\to\infty$ and $c_j\to c\ge\gamma$.
Then for each fixed $\epsilon\ge0$,
\begin{equation}\label{Eq:fractional_limit}
 \Delta_q((x_j,c_j),\epsilon)\longrightarrow
 G_q(c,\epsilon):=c^{q-1}\left[
 \xi c\left(1+q-\frac{qc}{\gamma}\right)-(p-q)\epsilon
 \right].
\end{equation}
\end{lemma}
\begin{proof}
Fix $\epsilon\ge0$. Define the sequence $(t_j)_{j\ge1}$
by $t_j=t_{(x_j,c_j)}(\epsilon)$, the time needed to reach the same
residual $\epsilon$ from starting state $(x_j,c_j)$, and set
$r_j=1+t_j/c_j$. The excess part of $\Phi$ is
$\Phi(t,x,c)-\mu t=\int_0^t x(1+u/c)^{-p}\,\D u$. Its integrand
decreases in $u$, so bounding it by its values at $u=t$ and $u=0$ gives
\begin{equation}\label{Eq:Phi_sandwich}
 xt(1+t/c)^{-p}\le \Phi(t,x,c)-\mu t\le xt .
\end{equation}

We first show that the duration $t_j$ for the fixed residual $\epsilon$
tends to zero as the excitation level $x_j$ tends to infinity. Fix
$t>0$. Since $c_j\ge\gamma$, the
lower bound in \eqref{Eq:Phi_sandwich} gives
$\Phi(t,x_j,c_j)\ge x_jt(1+t/\gamma)^{-p}$, which tends to $\infty$ as
$j\to\infty$ because $x_j\to\infty$. Hence
$\Phi(t,x_j,c_j)>\epsilon=\Phi(t_j,x_j,c_j)$ for large $j$, and
$t_j<t$ because $\Phi$ is strictly increasing in time. Thus, for every
$t>0$, $0\le t_j<t$ for all sufficiently large $j$; that is, $t_j\to0$.

Next, we show that $x_jt_j\to\epsilon$. On the short interval
$[0,t_j]$, the decay factor $(1+u/c_j)^{-p}$ stays close to one and
$\mu t_j$ is small, so $\epsilon=\Phi(t_j,x_j,c_j)$ is close to
$x_jt_j$. To make this precise, set $t=t_j$ in
\eqref{Eq:Phi_sandwich} and use $\Phi(t_j,x_j,c_j)=\epsilon$:
\[
 \epsilon-\mu t_j\le x_jt_j\le(\epsilon-\mu t_j)(1+t_j/c_j)^{p}.
\]
Since $t_j\to0$ and $c_j\ge\gamma$, both bounds tend to $\epsilon$.
Hence $x_jt_j\to\epsilon$ and $x_j(r_j-1)=x_jt_j/c_j\to\epsilon/c$.

We now express the update in terms of two ratios that tend to one. Let
\[
 R_j=r_j^{-p}+\frac{\xi}{x_j},\qquad
 D_j=r_j^{-(p+1)}+\frac{\xi c_j}{\gamma x_j}.
\]
Here $R_j$ is the ratio of the new excitation level to the old one,
and $D_j$ is the corresponding ratio for slope magnitudes, with the
factor $p$ canceling. In this
notation, \eqref{Eq:Psi_x} and \eqref{Eq:Psi_c} at $(x_j,c_j,t_j)$
become
\[
 \Psi_x=x_jr_j^{-p}+\xi=x_jR_j,\qquad
 \Psi_c=\frac{x_jR_j}{(x_j/c_j)D_j}=\frac{c_jR_j}{D_j},
\]
so the updated value of $V_q$ is
$\Psi_x\Psi_c^{\,q}=x_jc_j^qR_j^{1+q}D_j^{-q}$, and
\[
 \Delta_q((x_j,c_j),\epsilon)=x_jc_j^q\bigl(R_j^{1+q}D_j^{-q}-1\bigr).
\]
Because of the factor $x_j\to\infty$, the limit of $\Delta_q$ depends
on how fast $R_j$ and $D_j$ approach one. It therefore remains to find
the limits of $x_j(R_j-1)$ and $x_j(D_j-1)$.

We use the same first-order expansion for the decay factors
$r_j^{-a}$ and for the updated value of $V_q$. Let $f$ be
differentiable at a point $y_*$, and let $y_j$ satisfy
$x_j(y_j-y_*)\to L$. Since $x_j\to\infty$, $y_j\to y_*$, and
$f(y_j)-f(y_*)=\nabla f(y_*)\cdot(y_j-y_*)+o(|y_j-y_*|)$. Because
$x_j|y_j-y_*|$ is bounded, multiplying by $x_j$ gives
\begin{equation}\label{Eq:first_order}
 x_j\bigl(f(y_j)-f(y_*)\bigr)\to\nabla f(y_*)\cdot L .
\end{equation}

Apply \eqref{Eq:first_order} with $f(r)=r^{-a}$, $y_j=r_j$ and
$L=\epsilon/c$. This gives $x_j(r_j^{-a}-1)\to-a\epsilon/c$ for every
$a>0$. By the definitions of $R_j$ and $D_j$,
\[
 x_j(R_j-1)=x_j(r_j^{-p}-1)+\xi,\qquad
 x_j(D_j-1)=x_j(r_j^{-(p+1)}-1)+\frac{\xi c_j}{\gamma}.
\]
Taking $a=p$ and $a=p+1$, and using $c_j\to c$,
\[
 x_j(R_j-1)\to \xi-\frac{p\epsilon}{c},\qquad
 x_j(D_j-1)\to \frac{\xi c}{\gamma}-\frac{(p+1)\epsilon}{c}.
\]
Now apply \eqref{Eq:first_order} with $f(R,D)=R^{1+q}D^{-q}$,
$y_j=(R_j,D_j)$ and $y_*=(1,1)$, where $f(1,1)=1$ and
$\nabla f(1,1)=(1+q,-q)$:
\[
 x_j(R_j^{1+q}D_j^{-q}-1)
 \longrightarrow (1+q)\Bigl(\xi-\frac{p\epsilon}{c}\Bigr)
                  -q\Bigl(\frac{\xi c}{\gamma}-\frac{(p+1)\epsilon}{c}\Bigr)
 =\xi\Bigl(1+q-\frac{qc}{\gamma}\Bigr)-\frac{(p-q)\epsilon}{c},
\]
since $(1+q)p-q(p+1)=p-q$. Multiplying by $c_j^q\to c^q$ yields
\eqref{Eq:fractional_limit}.
\end{proof}

In $G_q$, the term $\xi c(1+q-qc/\gamma)$ comes from the jump at the
new event and $-(p-q)\epsilon$ from decay before it. For $q=p/(p+2)$,
the condition $\xi\gamma(p+1)<\mu_\varepsilon p^2$ gives
$\E[G_q(c,\varepsilon_1)]<0$ for every $c\ge\gamma$, so after averaging
over the residual the decay term dominates. Under this condition, the
Lyapunov function $W=\exp(\theta V_q)$ satisfies $PW\le\rho W$ outside
a compact set $K$ for some $\theta>0$ and $\rho<1$. Outside $K$, $W$
shrinks on average by the factor $\rho$ at each step, so the chain
keeps returning to $K$, and the time to return has a geometric tail.
This bound gives the geometric rate
in Theorem~\ref{Thm:FL_Joint}.

\begin{proposition}[Geometric drift]\label{Prop:geometric_drift}
Suppose $\mu_\varepsilon<\infty$ and
\begin{equation}\label{Eq:condition}
 \xi\gamma<\frac{\mu_\varepsilon p^2}{p+1}.
\end{equation}
For $q=p/(p+2)$, there exist $\theta>0$, $\rho\in(0,1)$, $b<\infty$,
and a nonempty compact rectangle $K\subset\mathcal S$ such that
$W=\exp(\theta V_q)$ satisfies
\begin{equation}\label{Eq:geometric_drift}
 PW(s)\le \rho W(s)+b\,\mathbbm{1}_K(s),\qquad s\in\mathcal S.
\end{equation}
\end{proposition}
\begin{proof}
Since $W(\Psi(s,t_s(\epsilon)))=W(s)\e^{\theta\Delta_q(s,\epsilon)}$,
\[
 \frac{PW(s)}{W(s)}=\E[\e^{\theta\Delta_q(s,\varepsilon_1)}].
\]
We show that $PW(s)/W(s)\le\rho$ for some $\rho<1$ outside a compact
rectangle $K=[\xi,L]\times[\gamma,C]$. We choose $C$, then $\theta$,
and then $L$.

We first treat large $c$. By the second bound in
Lemma~\ref{Lem:fractional_bounds}, $\Delta_q(s,\epsilon)\le-1$ for
every $\epsilon\ge0$ whenever $\xi c^q\ge2\xi(2\gamma)^q+1$. Fix $C$
satisfying this inequality; then $C>\gamma$. For $c\ge C$ and any
$\theta>0$,
\begin{equation}\label{Eq:drift_large_c}
 \E[\e^{\theta\Delta_q(s,\varepsilon_1)}]\le\e^{-\theta}.
\end{equation}

Next we bound the limit $G_q$ for $c\in[\gamma,C]$. The quadratic
$c(1+q-qc/\gamma)$ is at most $\gamma(1+q)^2/(4q)$. Therefore
\begin{equation}\label{Eq:Gq_bound}
 G_q(c,\epsilon)\le\kappa(c)(H_q-\epsilon),\qquad
 \kappa(c)=(p-q)c^{q-1},\qquad H_q=\frac{\xi\gamma(1+q)^2}{4q(p-q)}.
\end{equation}
With $q=p/(p+2)$, we have $H_q=\xi\gamma(p+1)/p^2$. This choice
minimizes $H_q$ over $q\in(0,1)$, since the logarithmic derivative of
$H_q$ has the sign of $(p+2)q-p$. Condition \eqref{Eq:condition} is
exactly $H_q<\mu_\varepsilon$, so $Z=H_q-\varepsilon_1$ satisfies
$\E Z<0$.

We next show that $\E Z<0$ implies $\E[\e^{uZ}]<1$ for small $u>0$.
To differentiate $\E[\e^{uZ}]$ at $u=0$, we bound the difference
quotient. By the mean value theorem, $\e^{uZ}-1=uZ\,\e^{\bar w}$ for
some $\bar w$ between $0$ and $uZ$. For $0<u\le1$, $uZ\le H_q$ because
$Z\le H_q$ and $H_q>0$, so $\bar w\le H_q$ and
\[
 \Bigl|\frac{\e^{uZ}-1}{u}\Bigr|\le |Z|\,\e^{H_q}.
\]
Moreover, for each value of $Z$,
\[
 \lim_{u\downarrow0}\frac{\e^{uZ}-1}{u}=Z.
\]
The bound $|Z|\e^{H_q}$ does not depend on $u$ and is integrable,
since $\E|Z|<\infty$. Dominated convergence therefore gives
$(\E[\e^{uZ}]-1)/u\to\E Z$ as $u\downarrow0$.
Since $\E Z<0$, the limit $\E Z$ lies strictly below $\E Z/2$, so
$(\E[\e^{uZ}]-1)/u\le\E Z/2$ for all small $u>0$. Multiplying
by $u$, we obtain $u_0\in(0,1]$ such that
\begin{equation}\label{Eq:mgf_bound}
 \E[\e^{uZ}]\le1+\tfrac12 u\,\E Z,\qquad 0<u\le u_0.
\end{equation}
Since $q<1$, $\kappa(c)=(p-q)c^{q-1}$ decreases in $c$. Choose
$\theta>0$ small enough that $\theta\kappa(\gamma)\le u_0$. Then
\[
 \theta\kappa(C)\le\theta\kappa(c)\le\theta\kappa(\gamma)\le u_0,
 \qquad c\in[\gamma,C],
\]
so \eqref{Eq:mgf_bound} holds with $u=\theta\kappa(c)$. Multiplying
\eqref{Eq:Gq_bound} by $\theta$ gives
$\theta G_q(c,\epsilon)\le\theta\kappa(c)(H_q-\epsilon)$, and with
$\epsilon=\varepsilon_1$ the right side is $\theta\kappa(c)Z$. Hence
\[
 \E[\e^{\theta G_q(c,\varepsilon_1)}]\le\E[\e^{\theta\kappa(c)Z}]
 \le1+\tfrac12\theta\kappa(c)\,\E Z
 \le1+\tfrac12\theta\kappa(C)\,\E Z=:\zeta<1 .
\]

The bound $\zeta$ holds for the limit $G_q$, while the drift condition
concerns $\Delta_q$ at finite $x$. We claim that, for large $x$, the
bound holds with $\zeta$ replaced by $(1+\zeta)/2$, which still lies
below one and leaves room for the limit: there is $L\ge\xi$ such that
\begin{equation}\label{Eq:drift_large_x}
 \E[\e^{\theta\Delta_q(s,\varepsilon_1)}]\le\frac{1+\zeta}{2},
 \qquad x\ge L,\ c\in[\gamma,C].
\end{equation}
Suppose not. Then there are states $(x_j,c_j)$ with $x_j\to\infty$,
$c_j\in[\gamma,C]$, and expectations above $(1+\zeta)/2$. Passing to a
subsequence, $c_j\to c\in[\gamma,C]$. By
Lemma~\ref{Lem:fractional_limit},
$\Delta_q((x_j,c_j),\epsilon)\to G_q(c,\epsilon)$ for each $\epsilon$.
The first bound in Lemma~\ref{Lem:fractional_bounds} gives
$\e^{\theta\Delta_q}\le\e^{\theta\xi\gamma^q}$, so dominated
convergence gives
$\E[\e^{\theta\Delta_q((x_j,c_j),\varepsilon_1)}]
\to\E[\e^{\theta G_q(c,\varepsilon_1)}]\le\zeta$, a contradiction.

Finally, a state outside $K$ has either $c>C$, or $c\le C$ and $x>L$.
By \eqref{Eq:drift_large_c} and \eqref{Eq:drift_large_x},
$PW(s)\le\rho W(s)$ there, with
$\rho=\max\{\e^{-\theta},(1+\zeta)/2\}<1$. On $K$, the first bound in
Lemma~\ref{Lem:fractional_bounds} gives
$PW(s)\le\e^{\theta\xi\gamma^q}W(s)\le b:=\e^{\theta(\xi\gamma^q+LC^q)}$.
Together these give \eqref{Eq:geometric_drift}.
\end{proof}

\subsection{Minorization and regularity}\label{Sect:minorization}

The drift bound brings the chain back to the compact rectangle $K$,
with a return time that has a geometric tail. Returning to $K$ is not
enough by itself. The chain must also forget its starting point:
from all starting states in $K$, the distribution a few steps later
must have a common nonzero lower bound. Formally, $K$ must be petite. Theorem~\ref{Thm:MT_harris} also
requires the sublevel sets $\{W\le n\}$ to be petite. Both kinds of
sets are compact, and small sets are petite, so it suffices to show
that compact sets are small. We also show that the chain is irreducible, as both
theorems assume, and aperiodic, as Theorem~\ref{Thm:MT_geometric}
assumes.

The next two lemmas show that compact sets are small. In the first, from every state near
$s_*$, the distribution of the state two events later is bounded below
by a common nonzero measure $\eta$. A lower bound of this kind is called
a minorization. Because $\eta$ does not depend on the starting state,
the chain forgets where it started with probability at least
$\eta(\mathcal S)$. In the second, from every state in a compact set,
one long duration, which erases most of the old excitation,
reaches a neighborhood of $s_*$ with probability bounded away from
zero. Together, the distribution of the state three events later is
bounded below by a fixed multiple of $\eta$ on the compact set.

\begin{lemma}[Two-step minorization near $s_*$]\label{Lem:two_step}
Under Assumption~\ref{Ass:residual_density}, there exist a relatively
open neighborhood $U_*$ of $s_*=(\xi,\gamma)$ and a nonzero finite
measure $\eta$ on $\mathcal S$ such that
\begin{equation}\label{Eq:minorization_two}
 P^2(s,B)\ge\eta(B),\qquad s\in U_*,\quad B\in\mathcal B(\mathcal S).
\end{equation}
\end{lemma}
\begin{proof}
Our goal is to find a neighborhood $U_*$ of $s_*$ and an open set
$B_*$ such that, for every starting state $s\in U_*$, the distribution
of the state two events later has a density on $B_*$ that is bounded
below by a positive constant not depending on $s$. Such a bound gives
$\eta$ directly. The argument is a direct version of the rank
condition for nonlinear state space models in
\citet[Chapter~7]{meyn2009markov}.

\emph{Why two steps.}
For a fixed starting state, the next state depends only on the
duration, so it lies on the curve $\{\Psi(s,t):t\ge0\}$. Its
distribution therefore has no density on the two-dimensional state
space, even when the residual has a positive density. Two steps use
two independent residuals and hence two durations. The next part
checks that varying the two durations moves the endpoint in two
independent directions.

\emph{Jacobians in the decay ratios.}
The computation is simplest when each duration is described by its
decay ratio. In the coordinates $(x,y)=T(x,c)$, where $T(x,c)=(x,x/c)$, a step
with decay ratio $u=(1+t/c)^{-1}\in(0,1)$,
equivalently $t=c(u^{-1}-1)$, is given by the affine map
\[
 R_u(x,y)=(\xi+xu^p,\ \xi/\gamma+yu^{p+1}).
\]
For two successive ratios $(u,v)$, the output is $z=R_v(R_u(x,y))$.
Its Jacobian with respect to the two ratios is
\[
 \frac{\partial z}{\partial(u,v)}=
 \begin{pmatrix}
 pxu^{p-1}v^p & pv^{p-1}(\xi+xu^p)\\
 (p+1)yu^pv^{p+1} & (p+1)v^p(\xi/\gamma+yu^{p+1})
 \end{pmatrix}.
\]
Its determinant is $p(p+1)\xi u^{p-1}v^{2p}(x/\gamma-yu)>0$
since $y\le x/\gamma$ and $u<1$.

The randomness enters through the residuals, not the ratios, so we also
need the Jacobian of the residuals with respect to the ratios. The first
residual is $\Phi(c(u^{-1}-1),x,c)$, which depends only on $u$. The
second is $\Phi(c_1(v^{-1}-1),x_1,c_1)$, where $(x_1,c_1)$ is the first
updated state. With $t_1=c(u^{-1}-1)$, $t_2=c_1(v^{-1}-1)$, and
$\partial_t\Phi=\phi$, the Jacobian of the two residuals with respect to
$(u,v)$ is
\[
 \frac{\partial(\epsilon_1,\epsilon_2)}{\partial(u,v)}=
 \begin{pmatrix}
 -\frac{c}{u^2}\,\phi(t_1;x,c) & 0\\
 \ast & -\frac{c_1}{v^2}\,\phi(t_2;x_1,c_1)
 \end{pmatrix},
\]
where $\ast$ is an entry that does not affect the determinant. Since
$\phi\ge\mu>0$, the determinant
$c\,c_1u^{-2}v^{-2}\phi(t_1;x,c)\,\phi(t_2;x_1,c_1)$ is positive.

\emph{Local inverse.}
We now combine the two Jacobians and invert the map from residuals to
the two-step state near a reference point. Write $\Gamma(s,\epsilon)=\Psi\bigl(s,t_s(\epsilon)\bigr)$ for the
state after one event with residual $\epsilon$. The state two events
after $s$, in the coordinates $(x,y)$, is
\[
 \Gamma^{(2)}(s,\epsilon_1,\epsilon_2)
 =T\Bigl(\Gamma\bigl(\Gamma(s,\epsilon_1),\epsilon_2\bigr)\Bigr).
\]
Fix the reference input at $s_*$ with ratios $(u,v)=(1/2,1/2)$; any
pair in $(0,1)^2$ would do. Let $(\epsilon_1^*,\epsilon_2^*)$ be the
corresponding residuals and $z^*=\Gamma^{(2)}(s_*,\epsilon_1^*,\epsilon_2^*)$.
Because $s_*$ is a corner of $\mathcal S$, we extend the update
formulas, which remain smooth for $x$ and $c$ slightly below $\xi$ and
$\gamma$, to an open neighborhood of $(s_*,\epsilon_1^*,\epsilon_2^*)$
in $\mathbb R^4$. By the chain rule, the Jacobian of the two-step state
with respect to the residuals is
\[
 \frac{\partial z}{\partial(\epsilon_1,\epsilon_2)}
 =\frac{\partial z}{\partial(u,v)}
 \left(\frac{\partial(\epsilon_1,\epsilon_2)}{\partial(u,v)}\right)^{-1}.
\]
Both factors are nonsingular by the computations above, so the
Jacobian of $\Gamma^{(2)}(s_*,\cdot)$ at $(\epsilon_1^*,\epsilon_2^*)$ is
nonsingular.

The inverse function theorem, applied to
$(s,\epsilon_1,\epsilon_2)\mapsto\bigl(s,\Gamma^{(2)}(s,\epsilon_1,\epsilon_2)\bigr)$,
gives
open neighborhoods $U_*$ of $s_*$ and $B_*$ of $z^*$ and a smooth map
$(s,z)\mapsto E_s(z)$ with values near $(\epsilon_1^*,\epsilon_2^*)$
such that
\[
 \Gamma^{(2)}\bigl(s,E_s(z)\bigr)=z,\qquad (s,z)\in U_*\times B_*.
\]
Here $U_*$ is a set of starting states and $B_*$ a set of states two
events later. Near $(\epsilon_1^*,\epsilon_2^*)$, $E_s(z)$ is the only
residual pair with this property.

Choosing
$U_*$ and $B_*$ smaller if necessary, we may assume that the closure of
$U_*\times B_*$ is compact and that $E_s(z)$ is defined and smooth on an
open set containing it. This lets us bound the density below by a
positive constant. We then intersect $U_*$ with
$\mathcal S$. The set $B_*$ can be taken inside $T(\mathcal S)$, since
the state $T^{-1}(z^*)$ has $x_2>\xi$ and $c_2>\gamma$.

\emph{Common lower bound.}
For $s\in U_*$, the residual pairs near $(\epsilon_1^*,\epsilon_2^*)$
alone give the state two events later, at $z\in B_*$, the density
below, by the change-of-variables formula. Other residual pairs can
only add to it, so the density is at least
\[
 f_\varepsilon(\epsilon_1)f_\varepsilon(\epsilon_2)
 |\det \mathrm D_z E_s(z)|,\qquad (\epsilon_1,\epsilon_2)=E_s(z).
\]
This function is continuous and positive on the compact closure of
$U_*\times B_*$, so it has a positive minimum $m$. For every $s\in U_*$
and every Borel set $B\subseteq\mathcal S$,
\[
 P^2(s,B)\ge m\,\mathrm{Leb}\bigl(T(B)\cap B_*\bigr)=:\eta(B).
\]
Since $B_*$ is a nonempty bounded open set, $\eta$ is nonzero and
finite, which proves \eqref{Eq:minorization_two}.
\end{proof}

\begin{lemma}[Compact sets are small]\label{Lem:foster_small}
Under Assumption~\ref{Ass:residual_density}, every relatively open
neighborhood $U$ of $s_*$ satisfies $\inf_{s\in K}P(s,U)>0$ for each
nonempty compact $K\subset\mathcal S$. With $U_*$ and $\eta$ from
Lemma~\ref{Lem:two_step},
\begin{equation}\label{Eq:minorization_three}
 P^3(s,B)\ge P(s,U_*)\eta(B),\qquad s\in\mathcal S,\quad
 B\in\mathcal B(\mathcal S).
\end{equation}
In particular, for each nonempty compact $K\subset\mathcal S$,
$P^3(s,\cdot)\ge a_K\eta$ for all $s\in K$ with
$a_K=\inf_{s\in K}P(s,U_*)>0$, so $K$ is small.
\end{lemma}
\begin{proof}
Let $\bar x=\max_{(x,c)\in K}x$. By the upper bound in
\eqref{Eq:Phi_sandwich}, $\Phi(t,x,c)\le(\mu+\bar x)t$ for
$(x,c)\in K$, so
\[
 t_s(\epsilon)\ge\frac{\epsilon}{\mu+\bar x},\qquad s\in K.
\]
Thus large residuals give long durations from every state in $K$. The
update $\Psi(s,t)$ tends uniformly on $K$ to $s_*$ as
$t\to\infty$, by \eqref{Eq:Psi_x}--\eqref{Eq:Psi_c}.
Combining the two, for any relatively open neighborhood $U$ of $s_*$, there is a
common interval of sufficiently large residuals sending all $s\in K$
into $U$. Its residual probability is strictly positive. Apply this
with $U=U_*$, the neighborhood of $s_*$ from
Lemma~\ref{Lem:two_step} on which the two-step bound holds. Then
$P(s,U_*)$ is at least this probability for every $s\in K$, so
$a_K=\inf_{s\in K}P(s,U_*)>0$.
By the kernel composition formula, \eqref{Eq:minorization_two} gives
\eqref{Eq:minorization_three}. On $K$ the lower bound is at least $a_K\eta(B)$, so $K$ is small.
\end{proof}

Every Borel set $B$ with $\eta(B)>0$ is reached with positive
probability from every state, in three steps and also in four steps.
Reaching these sets in three steps gives irreducibility. Reaching
them in both three and four steps rules out periodicity.

\begin{corollary}[Irreducibility and aperiodicity]\label{Cor:kernel_regularity}
Under Assumption~\ref{Ass:residual_density}, the chain is
$\psi$-irreducible for a nonzero finite measure $\psi$ and aperiodic.
\end{corollary}
\begin{proof}
Use $U_*$ and $\eta$ from Lemma~\ref{Lem:two_step}.
By Lemma~\ref{Lem:foster_small} with $K=\{s\}$, $P(s,U_*)>0$ for every
$s$. Hence, by \eqref{Eq:minorization_three}, any
Borel set $B$ with $\eta(B)>0$ has $P^3(s,B)>0$ from every state.
This gives $\eta$-irreducibility.

For any Borel $B$ with $\eta(B)>0$, integrating $P^3(u,B)>0$
against $P(s,\D u)$ gives $P^4(s,B)>0$ for every $s$.
If the chain had period $d>1$, its cyclic classes would cover the
state space up to a null set for a maximal irreducibility measure,
and hence up to an $\eta$-null set. Choose a cyclic class $E_i$ with
$\eta(E_i)>0$ and $s\in E_i$. Since both $P^3(s,E_i)$ and $P^4(s,E_i)$
are positive, cyclicity would force $d$ to divide both $3$ and $4$.
Thus the chain is aperiodic.
\end{proof}

\subsection{Positive Harris recurrence and geometric convergence}\label{Sect:positive_harris}

The drift bound and the minorization now give the main stability result.
They ensure a unique long-run distribution of activity and decay
for the model. The result also controls how the effect of the
initial state fades as the number of events grows.

\begin{theorem}\label{Thm:FL_Joint}
Under Assumption~\ref{Ass:residual_density}, if
$\mu_\varepsilon<\infty$ and \eqref{Eq:condition} holds, the chain
$(X_n,c_n)$ is positive Harris recurrent and has a unique invariant
probability measure $\pi$. For $W$ from Proposition~\ref{Prop:geometric_drift},
$\pi(W)<\infty$, and there are constants $M<\infty$ and $\beta\in(0,1)$
such that, for every $s\in\mathcal S$ and $n\ge0$,
\begin{equation}\label{Eq:geometric_tv}
 \|P^n(s,\cdot)-\pi\|_{\mathrm{TV}}\le M W(s)\beta^n.
\end{equation}
\end{theorem}
\begin{proof}
Use $W$, $\rho$, $b$, and $K$ from Proposition~\ref{Prop:geometric_drift}.
The set $K$ is small by Lemma~\ref{Lem:foster_small}, hence petite,
and the chain is $\psi$-irreducible and aperiodic by
Corollary~\ref{Cor:kernel_regularity}.

\emph{Harris recurrence.} Each sublevel set $\{W\le n\}$ is a sublevel
set of $V_q$, hence compact, and so petite by
Lemma~\ref{Lem:foster_small}. For $s\notin K$,
\eqref{Eq:geometric_drift} gives $PW(s)\le\rho W(s)\le W(s)$. Theorem~\ref{Thm:MT_harris} with $V=W$ and $C=K$
shows that the chain is Harris recurrent.

\emph{Geometric convergence.} The drift bound \eqref{Eq:geometric_drift}
is the drift condition of Theorem~\ref{Thm:MT_geometric} with $V=W$,
$\delta=1-\rho$, and $C=K$. Since $W$ is finite everywhere, the theorem
gives an invariant probability measure $\pi$ and constants $r>1$ and
$R<\infty$ with $\|P^n(s,\cdot)-\pi\|_W\le R\,W(s)r^{-n}$ for every $s$
and $n$. Together with Harris recurrence, the chain is positive Harris
recurrent. Since $W\ge1$, the total variation norm is at most the
$W$-norm, which gives \eqref{Eq:geometric_tv} with $M=R$ and
$\beta=1/r$.

\emph{Uniqueness and $\pi(W)<\infty$.} A recurrent
$\psi$-irreducible chain has an invariant measure that is unique up to
a constant multiple \citep[Theorem~10.0.1]{meyn2009markov}, so $\pi$ is
the unique invariant probability measure. For fixed $s$ and $n$,
iterating \eqref{Eq:geometric_drift} shows that $P^nW(s)$ is finite. For each $m\ge1$, the bounded function $W\wedge m$ satisfies
$|W\wedge m|\le W$, so
\[
 \pi(W\wedge m)\le P^nW(s)+\|P^n(s,\cdot)-\pi\|_W
 \le P^nW(s)+R\,W(s)r^{-n}.
\]
Letting $m\to\infty$ and using monotone convergence gives
$\pi(W)<\infty$.
\end{proof}

For duration forecasts, a direct implication is that the probability
of a duration exceeding $t$, averaged over the state at event $n$,
approaches the stationary average
$\int \bar F_\varepsilon(\Phi(t,s))\,\pi(\D s)$.
The conditional survival probability takes values in $[0,1]$, so
\eqref{Eq:geometric_tv} bounds the difference by the same geometric
rate.

\paragraph{Comparison with the Hawkes subcriticality condition.}
With $\xi=\alpha/\gamma^p$, the kernel in Section~\ref{Sect:motive}
has branching ratio
\[
 \int_0^\infty g(u)\,\D u
 =\frac{\alpha\gamma^{1-p}}{p-1}
 =\frac{\xi\gamma}{p-1}.
\]
The classical subcriticality condition for a stationary linear
Hawkes process is therefore $\xi\gamma<p-1$ \citep{Hawkes&Oakes}.
For unit exponential residuals, $\mu_\varepsilon=1$ and
Theorem~\ref{Thm:FL_Joint} gives the sufficient condition
$\xi\gamma<p^2/(p+1)$ for recurrence of our embedded chain.
Since $p^2/(p+1)=p-1+1/(p+1)$, the latter bound includes all matched
subcritical Hawkes parameter values and some beyond them. The two
conditions apply to their respective models. For other admissible
residual distributions, our sufficient region is
$\xi\gamma<\mu_\varepsilon p^2/(p+1)$.
Outside this region the bound leaves recurrence and convergence
undetermined.

\FloatBarrier
\begingroup
\renewcommand{\topfraction}{0.9}
\renewcommand{\textfraction}{0.1}
\renewcommand{\floatpagefraction}{0.7}
\section{Empirical study}\label{Sect:empirical}

\subsection{One-tick durations and MPL estimates}

We first describe the estimation results of the proposed Markov
power-law (MPL) model and then examine goodness of fit. The
observations are durations between one-tick midquote moves in CME E-mini S\&P 500 futures (E-mini) and
Nasdaq NVDA. The midquote is the average of the best bid and ask.
An event occurs when it first moves by at least one native tick
from its value at the preceding event. A tick is 0.25 index points
for E-mini (\$12.50 per contract) and \$0.01 per share for NVDA.
The data cover May--June 2025. Both samples are restricted to
09:30--16:00 New York time, the regular US equity session;
the E-mini sample excludes the rest of its nearly 24-hour trading day.

We examine the effect of the residual distribution by comparing unit
exponential residuals (Exp) with a mean-one
mixture of two exponentials (MixExp), the mixture law used for ACD
residuals by \citet{deluca2004mixture} and
\citet{yatigammana2019forecasting};
Section~\ref{Sect:exp_mixexp} reports the comparison. MixExp has two parameters: the
weight $w$ of the shorter-mean component and the ratio $r>1$ of the
two component means. All residual laws have mean one, so the rate
parameters have the same normalization under each law;
Appendix~\ref{App:residuals} gives their densities.

We estimate the rate parameters $(\mu,\xi,\gamma,p)$ and the residual
parameters $(w,r)$ of MPL--MixExp jointly, with $p\in[1.001,20]$.
Each fit uses the preceding 5,000 durations and is updated every
100 observations. The state updates after each observed duration
between refits. For each date, we take the median of each parameter
over that date's rolling estimates; Table~\ref{tab:basic_free_parameters}
reports the median of these daily medians across dates.

\begin{table}[!tbp]
\centering\footnotesize
\setlength{\tabcolsep}{4pt}
\caption{MPL--MixExp estimates for one-tick durations (medians of
daily medians). $\mu$ and $\xi$ are in s$^{-1}$; $\gamma$ is in
seconds. Lower and Upper give mean daily percentages of estimates at
the bounds 1.001 and 20 for $p$. Max.\ $R$ is the largest $R$ over
all rolling estimates.}
\label{tab:basic_free_parameters}
\begin{tabular}{lrrrrrrrrr}
\toprule
Market & $\mu$ & $\xi$ & $\gamma$ & $p$ & $w$ & $r$ & Lower (\%) & Upper (\%) & Max. $R$ \\
\midrule
E-mini & 2.34 & 19.4 & $1.09\times10^{-2}$ & 1.41 & 0.272 & 66.6 & 13.1 & 11.0 & 0.932 \\
NVDA & 1.64 & 18.2 & $7.72\times10^{-3}$ & 1.001 & 0.153 & 57.6 & 52.5 & 0.1 & 0.888 \\
\bottomrule
\end{tabular}
\end{table}

The median baseline rate $\mu$ is 2.34 s$^{-1}$ for E-mini and
1.64 s$^{-1}$ for NVDA. Each event raises the excitation level by
$\xi$, about 19.4 and 18.2 s$^{-1}$, roughly 8 and 11 times the
baseline rate. The new excitation has decay scale $\gamma$ of about
11 and 8 milliseconds. In the residual law, the shorter-mean
component has weight 0.272 for E-mini and 0.153 for NVDA, and the longer
component mean is 66.6 and 57.6 times the shorter one.

The median of the daily median exponent $p$ estimates is 1.410
(interquartile range [1.177, 1.759]) for E-mini and 1.001
([1.001, 1.241]) for NVDA. For NVDA, the
median coincides with the lower search bound; the mean daily
proportion of estimates at that bound is 52.5\%. For mean-one
residuals, the stability condition~\eqref{Eq:condition} becomes $R<1$,
where $R=\xi\gamma(p+1)/p^2$. It is not imposed in
estimation, and every estimate reported here satisfies it.

\subsection{Goodness of fit}

We evaluate fit within the estimation windows. For each date, we
take consecutive nonoverlapping 5,000-duration training windows and
evaluate every duration in a window with the parameters fitted to
that window, starting from the same initial state as in estimation.
The resulting residuals are pooled within each date. The diagnostic
samples contain 3,465,000 E-mini durations on 42 dates and 2,355,000 NVDA
durations on 41 dates. Because the parameters are fitted to the same
observations, the PIT distance below is a descriptive measure of fit,
not a Kolmogorov--Smirnov test. Out-of-sample forecasts are evaluated
separately in Section~\ref{Sect:prediction_ranges}.

For each observed duration, we evaluate its fitted conditional
cumulative probability. These values are the probability integral
transform (PIT); a correct conditional distribution makes them
uniform on $(0,1)$. We also transform them into exponential
residuals, $Z_n=-\log(1-U_n)$, where $U_n$ is the PIT. Equivalently,
$Z_n$ is the integrated conditional intensity over the observed
duration; it equals the integrated latent rate
$\Phi(\tau_n,X_{n-1},c_{n-1})$ only under Exp. The reference
mean and variance of $Z_n$ are one. The PIT Kolmogorov distance is the
largest distance between the empirical PIT CDF and the uniform
CDF. ACF RMS is the root mean square of exponential-residual
autocorrelations at lags 1--50. Each diagnostic is computed by
date and then averaged over dates with equal weights; we call this
average the date-balanced mean.

\subsubsection{Fixed versus free $p$}\label{Sect:fixed_free_p}

We next examine whether fixing $p$ materially changes the one-tick
fit. The free-$p$ estimates in Table~\ref{tab:basic_free_parameters}
often stop at a search bound: 52.5\% of NVDA estimates at 1.001 and
11.0\% of E-mini estimates at 20. A free exponent also adds a parameter that
SE and ACD do not have in the comparisons below. We fix $p=1.02$ as a common specification for all samples and do not
select it on the May--June sample; in the one-tick samples,
Table~\ref{tab:basic_fixed_free_gof} shows that this costs little
relative to estimating $p$ on the same windows and diagnostic
observations.

\begin{table}[!tbp]
\centering\small
\caption{Fixed versus estimated exponents in one-tick MPL--MixExp.
Entries are date-balanced means. The fixed specification
uses $p=1.02$; the estimated exponent ranges over $[1.001,20]$.
PIT distance and ACF RMS have reference value zero; the reference
residual mean and variance are one.}
\label{tab:basic_fixed_free_gof}
\begin{tabular}{llrrrr}
\toprule
Market & Exponent & PIT distance & ACF RMS & $\overline Z$ & $\Var(Z)$ \\
\midrule
E-mini & Fixed & 0.04285 & 0.02970 & 1.002 & 1.093 \\
E-mini & Estimated & 0.04270 & 0.03017 & 1.002 & 1.119 \\
NVDA & Fixed & 0.01011 & 0.03967 & 1.000 & 1.078 \\
NVDA & Estimated & 0.01005 & 0.04036 & 1.000 & 1.084 \\
\bottomrule
\end{tabular}
\end{table}

Estimating $p$ reduces the PIT distance from 0.04285 to 0.04270
for E-mini and from 0.01011 to 0.01005 for NVDA, but increases residual
variance and autocorrelation. Thus fixing $p=1.02$ increases the PIT
distance only slightly in these one-tick MixExp samples, while using
one fewer parameter.

\subsubsection{Exp versus MixExp}\label{Sect:exp_mixexp}

We therefore hold $p=1.02$ for the following Exp--MixExp comparison.
The remaining rate and residual parameters are reestimated under
each law. Table~\ref{tab:basic_parameters} reports their estimates.

\begin{table}[!tbp]
\centering\small
\caption{MPL estimates for one-tick durations, fixing $p=1.02$.
The units of $\mu$ and $\xi$ are s$^{-1}$; $\gamma$ is in seconds.
The mixture parameters $w$ and $r$ apply only to MixExp. Entries
and median $R$ summarize daily medians, where
$R=\xi\gamma(p+1)/p^2$.}
\label{tab:basic_parameters}
\begin{tabular}{llrrrrrr}
\toprule
Market & Residual & $\mu$ & $\xi$ & $\gamma$ & $w$ & $r$ & Med. $R$ \\
\midrule
E-mini & Exp & 1.56 & 473 & $2.15\times10^{-4}$ & -- & -- & 0.206 \\
E-mini & MixExp & 2.18 & 20.4 & $4.89\times10^{-3}$ & 0.271 & 64.6 & 0.193 \\
NVDA & Exp & 1.68 & 174 & $4.24\times10^{-4}$ & -- & -- & 0.160 \\
NVDA & MixExp & 1.58 & 16.8 & $5.20\times10^{-3}$ & 0.158 & 57.0 & 0.207 \\
\bottomrule
\end{tabular}
\end{table}

The residual law changes the fitted rate parameters. From Exp to
MixExp, $\xi$ falls from 473 to 20.4 for E-mini and from 174 to 16.8 for
NVDA, while $\gamma$ rises from $2.15\times10^{-4}$ to
$4.89\times10^{-3}$ seconds and from $4.24\times10^{-4}$ to
$5.20\times10^{-3}$ seconds. Under Exp, the fitted jumps are larger
and the decay scales of new contributions are shorter. Under MixExp,
a residual comes with probability 0.271 for E-mini and 0.158 for NVDA
from a component whose mean is about 1/65 and 1/57 of the other
component mean. Very short durations can then arise from small
residuals, without the large and short-lived rate jumps that Exp
requires. The
median of $\xi\gamma$, computed within each fitted vector before
aggregation across dates, changes from 0.106 to 0.0995 for E-mini and
from 0.0823 to 0.107 for NVDA. Every fitted vector satisfies $R<1$.

Replacing Exp by MixExp improves the fit in both markets (one-tick
MPL rows of Table~\ref{tab:extended_main} in
Section~\ref{Sect:extended_gof}). The PIT distance falls by 29\%
for E-mini, from 0.0600 to 0.0429, and by 37\% for NVDA, from 0.0161 to
0.0101. The residual ACF RMS falls by 22\% and 36\%, from 0.0380 to
0.0297 and from 0.0624 to 0.0397. The one-step negative log predictive scores in
Section~\ref{Sect:prediction_ranges} show the same gain, which is
larger than the MPL--SE difference under MixExp. The residual variance
changes from 1.099 to 1.093 for E-mini and from 1.141 to 1.078 for NVDA,
remaining above one.

\subsection{Price thresholds, models, and residual laws}\label{Sect:extended_gof}

We extend the comparison to midquote displacements of
$\delta\in\{0.5,1,2\}$ native ticks in both markets. Each threshold
uses its own preceding event midquote as the anchor. Half a native
tick is a possible midpoint displacement: it is 0.125 index points
for E-mini and \$0.005 per share for NVDA. Larger thresholds give fewer events and longer
durations. Data construction and the trading session are as for the
one-tick samples (Appendix~\ref{App:basic_support}).
Table~\ref{tab:extended_samples} describes the retained May--June
samples. A date is retained when it has at least 1,000 durations
after the first training window, the requirement used for the
forecasts in Section~\ref{Sect:prediction_ranges}. Every specification within a market and
threshold uses the same dates and observations. Since windows
contain 5,000 durations, their clock-time lengths vary across
thresholds and dates.

\begin{table}[!tbp]
\centering\small
\caption{Diagnostic samples, May--June 2025. Available dates meet the
data requirements of Appendix~\ref{App:basic_support}. Durations are
in-sample diagnostic observations, from one to six windows per date
at two ticks. Median and 99th percentile are date-balanced means, in
seconds.}
\label{tab:extended_samples}
\begin{tabular}{lrrrrr}
\toprule
Market & Ticks & Dates retained/available & Durations & Median & 99th pct. \\
\midrule
E-mini & 0.5 & 43/43 & 9,790,000 & $1.654\times10^{-3}$ & 2.386 \\
E-mini & 1 & 42/42 & 3,465,000 & $4.445\times10^{-2}$ & 3.512 \\
E-mini & 2 & 29/36 & 275,000 & 1.093 & 14.52 \\
NVDA & 0.5 & 41/41 & 8,110,000 & $6.344\times10^{-4}$ & 2.198 \\
NVDA & 1 & 41/41 & 2,355,000 & 0.1082 & 4.275 \\
NVDA & 2 & 41/41 & 445,000 & 0.6393 & 12.30 \\
\bottomrule
\end{tabular}
\end{table}

We compare three dynamics:
\begin{itemize}
 \item MPL, with $p=1.02$ (Section~\ref{Sect:fixed_free_p}).
 \item The self-exciting exponentially decaying model (SE) of
 \citet{lee2026forecasting}, with latent rate
 $\phi_{n-1}(t)=\mu+X_{n-1}\exp(-t/\gamma)$ and an excitation jump
 $\xi$ after each event. With unit exponential residuals, SE is an
 exponential-kernel Hawkes process \citep{Hawkes2}.
 \item ACD(1,1) \citep{EngleRussell1998}, which models a conditional
 duration scale:
 \[
  \tau_n=\psi_n\varepsilon_n,\qquad
  \psi_n=\omega+a_1\tau_{n-1}+b_1\psi_{n-1},
  \quad \omega>0,\quad a_1,b_1\ge0,\quad a_1+b_1<1.
 \]
 With mean-one residuals, $\psi_n$ is the conditional mean duration.
\end{itemize}
MPL and SE propagate a latent rate in elapsed time; ACD updates
the duration scale in event time. All three models use the same
5,000-duration rolling windows and 100-observation refit interval.

For each of the three models, we fit Exp, Burr type XII, MixExp, and
generalized gamma residuals, all normalized to mean one
(Appendix~\ref{App:residuals}), giving twelve specifications. Exp adds no shape
parameter; each flexible law adds two. Since MPL fixes $p$, all
models estimate three parameters under Exp and
five under a flexible law. The residual argument is the integrated
latent rate for MPL and SE, and $\tau_n/\psi_n$ for ACD. Applying
the fitted residual CDF gives a common PIT and exponential-residual
reference for all twelve specifications. Search ranges are given in
Appendices~\ref{App:residuals} and~\ref{App:extended_gof}.

\subsection{Fit across thresholds and residual laws}\label{Sect:fit_thresholds}

The results extend Section~\ref{Sect:exp_mixexp} in three ways:
\begin{enumerate}
 \item MixExp lowers the mean PIT distance relative to Exp for every
 model and sample, and every flexible law lowers BIC
 relative to Exp in every case.
 \item At half and one tick, MPL and SE have smaller PIT distances
 than ACD under every residual law. At two ticks, the ordering depends
 on the market and residual law. Under MixExp, ACD has the smallest PIT
 distance for E-mini, whereas MPL and SE have smaller distances for NVDA.
 Between the two rate models, SE has smaller PIT distances under
 MixExp in four of six samples, while MPL has the lower BIC
 in all six.
 \item ACD leaves the least residual autocorrelation in nearly every
 specification, but low autocorrelation alone does not imply a
 correct conditional duration distribution.
\end{enumerate}

Table~\ref{tab:extended_main} reports date-balanced mean PIT
distances and residual ACF RMS at lags 1--50 under Exp and MixExp,
the two laws used in the forecast comparison;
Table~\ref{tab:extended_flex} in Appendix~\ref{App:extended_gof}
adds Burr and generalized gamma. Because diagnostic sample sizes
differ across thresholds, PIT distances and ACF RMS are not directly
comparable across thresholds.

Figure~\ref{fig:es_representative_gof} illustrates the common
MixExp law for one-tick E-mini on June 2, 2025. MPL has a smaller
PIT distance than ACD on this date, 0.048 against 0.071. Both Q--Q
plots lie above the Exp$(1)$ reference in the upper tail; the
departure is much larger for ACD. ACD therefore assigns too little
probability to very long durations, consistent with its larger
upper-tail exceedance errors in Section~\ref{Sect:prediction_ranges}.

\begin{figure}[!tbp]
\centering
\begin{subfigure}[t]{0.36\textwidth}
  \centering
  \includegraphics[width=\linewidth]{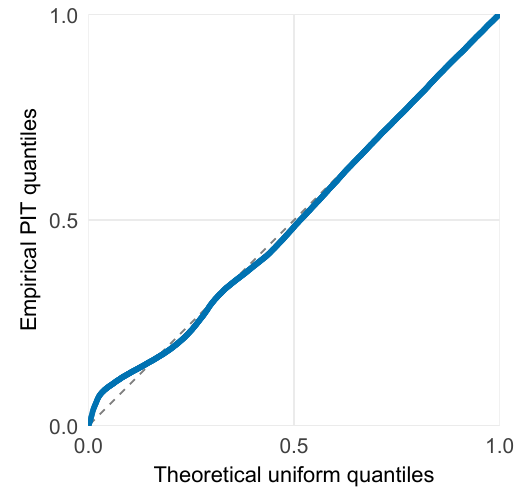}
  \caption{P--P: MPL ($p=1.02$).}
  \label{fig:es_pp_mpl}
\end{subfigure}\hspace{0.04\textwidth}
\begin{subfigure}[t]{0.36\textwidth}
  \centering
  \includegraphics[width=\linewidth]{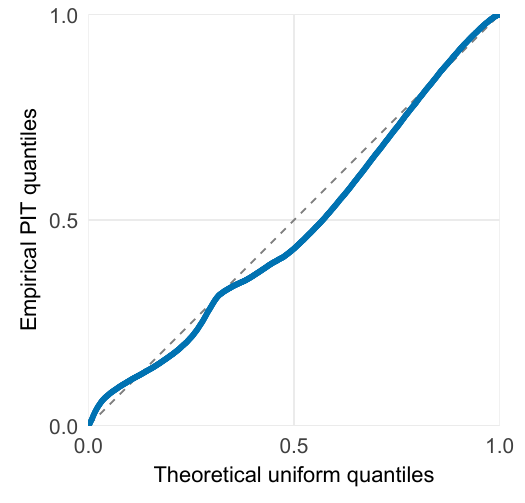}
  \caption{P--P: ACD(1,1).}
  \label{fig:es_pp_acd}
\end{subfigure}

\medskip
\begin{subfigure}[t]{0.36\textwidth}
  \centering
  \includegraphics[width=\linewidth]{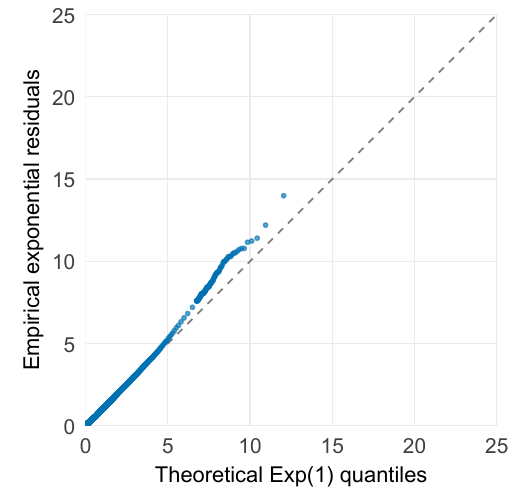}
  \caption{Q--Q: MPL ($p=1.02$).}
  \label{fig:es_qq_mpl}
\end{subfigure}\hspace{0.04\textwidth}
\begin{subfigure}[t]{0.36\textwidth}
  \centering
  \includegraphics[width=\linewidth]{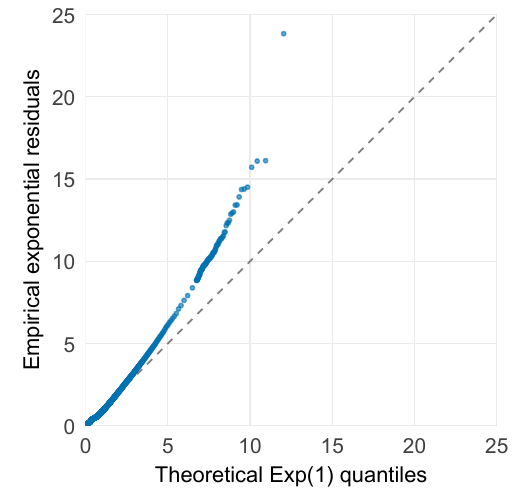}
  \caption{Q--Q: ACD(1,1).}
  \label{fig:es_qq_acd}
\end{subfigure}
\caption{E-mini one-tick diagnostics on June 2, 2025.
Both models use MixExp residuals and the same 85,000 in-sample
durations from 17 nonoverlapping training windows. P--P panels compare PIT quantiles with uniform
quantiles; Q--Q panels compare exponential-residual quantiles
with Exp$(1)$ quantiles. Dashed lines are the reference diagonals.
Points are deterministic quantile subsamples; the upper 100 ranks
are all retained.}
\label{fig:es_representative_gof}
\end{figure}

\begin{table}[!tbp]
\centering\footnotesize
\setlength{\tabcolsep}{3.5pt}
\caption{In-sample fit under Exp and MixExp over the selected
nonoverlapping training windows. Entries are date-balanced means of
the PIT Kolmogorov distance, the exponential-residual ACF RMS at lags
1--50, and $\Delta$BIC, the BIC minus the smallest BIC among all
twelve specifications in the same window. Lower is better. Bold entries are the lowest of the three
models for each statistic. Table~\ref{tab:extended_flex} in
Appendix~\ref{App:extended_gof} gives Burr and generalized gamma.}
\label{tab:extended_main}
\begin{tabular}{lrl rrr rrr rrr}
\toprule
 & & & \multicolumn{3}{c}{PIT distance} & \multicolumn{3}{c}{ACF RMS}
 & \multicolumn{3}{c}{$\Delta$BIC} \\
\cmidrule(lr){4-6}\cmidrule(lr){7-9}\cmidrule(lr){10-12}
Market & Ticks & Residual & MPL & SE & ACD & MPL & SE & ACD & MPL & SE & ACD \\
\midrule
E-mini & 0.5 & Exp & \textbf{0.0669} & 0.0763 & 0.5958 & 0.1083 & 0.0865 & \textbf{0.0358} & \textbf{804.8} & 2080.6 & 27889.8 \\
 &  & MixExp & 0.0644 & \textbf{0.0457} & 0.0946 & 0.1053 & 0.0878 & \textbf{0.0716} & \textbf{185.6} & 558.8 & 1582.3 \\
E-mini & 1 & Exp & \textbf{0.0600} & 0.0675 & 0.3500 & 0.0380 & 0.0395 & \textbf{0.0061} & \textbf{409.8} & 959.3 & 10647.1 \\
 &  & MixExp & 0.0429 & \textbf{0.0404} & 0.0732 & 0.0297 & 0.0310 & \textbf{0.0082} & \textbf{14.4} & 78.7 & 651.0 \\
E-mini & 2 & Exp & 0.0314 & 0.0299 & \textbf{0.0293} & 0.0240 & 0.0230 & \textbf{0.0121} & 267.1 & 308.0 & \textbf{179.5} \\
 &  & MixExp & 0.0293 & 0.0257 & \textbf{0.0230} & 0.0217 & 0.0220 & \textbf{0.0116} & 96.4 & 121.8 & \textbf{14.0} \\
NVDA & 0.5 & Exp & \textbf{0.0600} & 0.0993 & 0.5523 & 0.1553 & 0.1249 & \textbf{0.0560} & \textbf{498.8} & 2016.9 & 24348.4 \\
 &  & MixExp & \textbf{0.0263} & 0.0423 & 0.1210 & 0.1472 & 0.1289 & \textbf{0.1039} & \textbf{101.4} & 587.3 & 1854.5 \\
NVDA & 1 & Exp & \textbf{0.0161} & 0.0467 & 0.2456 & 0.0624 & 0.0602 & \textbf{0.0069} & \textbf{170.1} & 598.5 & 5690.4 \\
 &  & MixExp & \textbf{0.0101} & 0.0164 & 0.0503 & 0.0397 & 0.0378 & \textbf{0.0091} & \textbf{11.9} & 101.7 & 418.2 \\
NVDA & 2 & Exp & \textbf{0.0423} & 0.0519 & 0.0735 & 0.0573 & 0.0480 & \textbf{0.0122} & \textbf{630.4} & 757.0 & 848.3 \\
 &  & MixExp & 0.0111 & \textbf{0.0104} & 0.0267 & 0.0216 & 0.0219 & \textbf{0.0116} & 57.6 & 94.8 & \textbf{25.7} \\
\bottomrule
\end{tabular}
\end{table}

At half and one tick, the smallest PIT distance in each sample is
obtained by MPL or SE with a flexible residual law
(Tables~\ref{tab:extended_main} and \ref{tab:extended_flex}). At two ticks,
ACD--MixExp has the smallest distance for E-mini, 0.0230, whereas for
NVDA SE--MixExp and MPL--MixExp, at 0.0104 and 0.0111, lie well
below ACD--MixExp, at 0.0267. The MPL--ACD difference under MixExp is
$0.0063$ for E-mini ($[0.0012,0.0105]$) and $-0.0156$ for NVDA
($[-0.0183,-0.0128]$). These are 95\% paired circular block bootstrap
percentile intervals \citep{PolitisRomano1992}, using blocks of five
consecutive dates and 10,000 draws;
Appendix~\ref{App:extended_gof} gives the intervals for all samples
under MixExp.

Burr and generalized gamma raise the PIT distance relative to Exp in
four cases: one-tick E-mini under MPL, by 0.0030 and 0.0010, and
two-tick E-mini under ACD, by less than 0.0004 (Table~\ref{tab:extended_flex}).
In all four cases the same laws lower $\Delta$BIC and bring
the residual variance closer to one; for one-tick E-mini under MPL, it
falls from 1.099 under Exp to 0.977 and 0.953. Maximum likelihood
fits the whole density, so a better likelihood need not reduce the
largest PIT deviation.

At half and one tick, ACD--Exp has the smallest residual ACF RMS in
each sample, but its PIT distances range from 0.2456 to 0.5958, so low
residual autocorrelation alone does not establish a good fit. Under
the common MixExp law, MPL has smaller ACF RMS than SE
for one-tick E-mini, with a paired difference of $-0.0013$ (95\%
date-block bootstrap interval $[-0.0021,-0.0003]$). MPL also has smaller
mean ACF RMS at two ticks in both markets (0.0217 versus 0.0220 for
E-mini; 0.0216 versus 0.0219 for NVDA), although both paired intervals
include zero.
SE has smaller ACF RMS for both half-tick samples and for one-tick
NVDA. Both models retain more residual correlation than ACD--MixExp.

We also compare fitted conditional likelihoods with a penalty for
the number of parameters. BIC is computed in each selected
nonoverlapping training window, the same windows used for the
residual diagnostics. Within a window, we subtract the smallest
BIC among the twelve specifications. We then average within
each date and take the date-balanced mean.
Table~\ref{tab:extended_main} reports the resulting $\Delta$BIC for
Exp and MixExp; Table~\ref{tab:extended_flex} gives Burr and
generalized gamma.

MPL--MixExp has the smallest mean $\Delta$BIC among the twelve
specifications for one-tick E-mini and NVDA. At two ticks,
ACD--MixExp has the smallest criterion in both markets. At half
a tick, SE--Burr is preferred for E-mini and MPL--generalized gamma
for NVDA (Table~\ref{tab:extended_flex}).

We also compare fitted marginal duration quantiles with the
empirical ones (Appendix~\ref{App:extended_gof},
Table~\ref{tab:extended_q99}). The residual law with the smallest PIT
distance need not give the smallest quantile error. In half-tick
NVDA, MPL--MixExp reduces the PIT distance from 0.0600 under Exp to
0.0263, while the absolute relative error of the 99th percentile
rises from 3.4\% to 4.4\%. The smallest 99th-percentile error is
obtained by MPL--Exp for one-tick E-mini (3.5\%) and by MPL--generalized
gamma for two-tick E-mini and NVDA (3.3\% and 2.7\%).
\FloatBarrier
\subsection{One-step prediction}\label{Sect:prediction_ranges}

The goodness-of-fit comparison above is in-sample. We now evaluate
one-step forecasts out of sample: the parameters estimated on each
training window are held fixed for the next 100 durations, while the
state updates after each observed duration. We first compare the
overall accuracy of MPL, SE, and ACD under MixExp by the negative log
predictive score (NLS), the continuous ranked probability score
(CRPS), the root mean squared error (RMSE) and other point forecast
losses, and interval coverage. We then separate
the effects of the rate specification and the residual law, locate
density differences across duration ranges with censored log scores,
and evaluate forecast quantiles by pinball loss and upper-tail
exceedance frequencies. All parameters
are jointly fitted by conditional maximum likelihood, so the
contrasts compare complete fitted specifications.

Each forecast date has at least 1,000 forecast origins. E-mini
contributes 43, 42, and 29 dates at half,
one, and two ticks; NVDA contributes 41 dates at each
threshold. Scores are
averaged within each date and then summarized by the date-balanced
mean. The price threshold defines the event series; the duration
groups below are defined separately within each series.

The comparisons lead to five conclusions:
\begin{itemize}
 \item MPL has the lowest mean NLS at half and one tick and ACD at
 two ticks, whereas ACD has the lowest mean CRPS and RMSE in every
 sample, within 3\% of those of MPL. Lower values of these losses are
 better.
 \item For MPL, changing the residual law from Exp to MixExp lowers
 NLS more than changing the rate specification from SE to
 MPL under MixExp, in every sample except half-tick NVDA.
 \item Under MixExp, MPL has lower NLS than SE at every
 threshold in both markets.
 \item At half and one tick, the per-observation log-score gains of
 MPL over ACD in the realized-duration decomposition are largest among
 short durations. ACD has lower 90th-percentile loss.
 \item MPL has the upper-tail exceedance frequency closest to its
 target, although all three models exceed it.
\end{itemize}

\subsubsection{Overall forecast accuracy}\label{Sect:pred_overall}

This subsection compares MPL, SE, and ACD under MixExp residuals over
all forecast origins, before the later subsections examine the sources
and location of the differences.

Let $F_n$ and $f_n$ denote the conditional CDF and density of $\tau_n$
at its forecast origin. Then
$f_n(t)=\phi_{n-1}(t)f_\varepsilon(\Phi_{n-1}(t))$ for SE and MPL,
where $\Phi_{n-1}(t)=\int_0^t\phi_{n-1}(u)\,\D u$, and
$f_n(t)=\psi_n^{-1}f_\varepsilon(t/\psi_n)$ for ACD. We use the
following measures.
\begin{itemize}
 \item NLS: $\mathrm{NLS}_m=D^{-1}\sum_{d=1}^D s_{m,d}$, where $D$ is
 the number of dates and $s_{m,d}$ is the mean of $-\log f_n(\tau_n)$
 for model $m$ over the forecast origins of date $d$. A reduction of
 $\Delta>0$ corresponds to a date-balanced geometric mean density
 ratio of $\e^{\Delta}$.
 \item CRPS: $\int_0^\infty(F_n(t)-\mathbbm{1}_{\{\tau_n\le t\}})^2\,\D t$,
 in seconds.
 \item RMSE: root mean squared error of the predictive mean of
 $\tau_n$, in seconds.
 \item MAE: mean absolute error of the predictive median
 $F_n^{-1}(0.5)$, in seconds.
 \item LMSE: mean squared error of the predictive mean of $\log\tau_n$.
 \item Coverage: frequency with which $\tau_n$ falls in the central
 90\% predictive interval, from $F_n^{-1}(0.05)$ to $F_n^{-1}(0.95)$.
\end{itemize}
Lower values are better for the first five measures.
Table~\ref{tab:prediction_overall} reports date-balanced means under
MixExp.

\begin{table}[!tbp]
\centering\small
\caption{One-step forecast accuracy under MixExp residuals. Entries
are date-balanced means; RMSE is the square root of the
date-balanced mean squared error. Bold entries are the lowest of the
three models for each loss and, for coverage, the closest to 90\%.}
\label{tab:prediction_overall}
\begin{tabular}{lrl rrrrr r}
\toprule
Market & Ticks & Model & NLS & CRPS & RMSE & MAE & LMSE & Cover.\ (\%) \\
\midrule
E-mini & 0.5 & MPL & \textbf{$-$4.1893} & 0.1331 & 0.6596 & \textbf{0.1463} & 16.079 & \textbf{88.1} \\
 &  & SE & $-$4.1470 & 0.1316 & 0.6524 & 0.1466 & 15.768 & 86.1 \\
 &  & ACD & $-$4.0449 & \textbf{0.1298} & \textbf{0.6461} & 0.1465 & \textbf{15.242} & 81.4 \\
E-mini & 1 & MPL & \textbf{$-$1.4711} & 0.2653 & 0.8717 & \textbf{0.3401} & \textbf{9.131} & 92.3 \\
 &  & SE & $-$1.4608 & 0.2652 & 0.8704 & 0.3407 & 9.134 & 91.9 \\
 &  & ACD & $-$1.4064 & \textbf{0.2624} & \textbf{0.8581} & 0.3452 & 9.221 & \textbf{89.1} \\
E-mini & 2 & MPL & 1.6337 & 1.3518 & 3.5187 & 1.8408 & 2.583 & \textbf{89.5} \\
 &  & SE & 1.6376 & 1.3570 & 3.5313 & 1.8447 & 2.607 & 89.2 \\
 &  & ACD & \textbf{1.6275} & \textbf{1.3439} & \textbf{3.5104} & \textbf{1.8318} & \textbf{2.527} & 88.8 \\
NVDA & 0.5 & MPL & \textbf{$-$3.9473} & 0.1150 & 0.4847 & \textbf{0.1312} & 13.161 & \textbf{92.1} \\
 &  & SE & $-$3.8930 & 0.1153 & 0.4835 & 0.1313 & \textbf{13.128} & 87.9 \\
 &  & ACD & $-$3.7673 & \textbf{0.1144} & \textbf{0.4793} & 0.1313 & 13.131 & 81.5 \\
NVDA & 1 & MPL & \textbf{$-$0.6725} & 0.3075 & 0.9053 & 0.4091 & \textbf{6.886} & \textbf{88.5} \\
 &  & SE & $-$0.6618 & 0.3070 & 0.9029 & \textbf{0.4086} & 6.915 & 87.7 \\
 &  & ACD & $-$0.6386 & \textbf{0.3033} & \textbf{0.8879} & 0.4101 & 7.082 & 85.9 \\
NVDA & 2 & MPL & 1.6827 & 1.5526 & 3.9396 & 2.1163 & 3.216 & 88.7 \\
 &  & SE & 1.6882 & 1.5593 & 3.9560 & 2.1200 & 3.256 & 88.3 \\
 &  & ACD & \textbf{1.6694} & \textbf{1.5350} & \textbf{3.9205} & \textbf{2.0925} & \textbf{3.062} & \textbf{88.8} \\
\bottomrule
\end{tabular}
\end{table}

The two proper scores rank the models differently. MPL has the
lowest NLS at half and one tick in both markets and ACD at two
ticks, whereas ACD has the lowest CRPS and RMSE in every sample; its
means are within 3\% of those of MPL. The MAE and LMSE rankings vary
across samples. At two ticks, MPL has lower NLS, CRPS, RMSE, MAE, and
LMSE than SE in both markets. At half and one tick, SE has lower RMSE
in all four samples and lower CRPS in three. Coverage of the 90\% interval is
closest to 90\% for MPL in four of six samples, while ACD covers
about 81\% at half tick.

\subsubsection{Residual law and rate specification}\label{Sect:pred_residual_rate}

We next separate the two sources of NLS gains: the power-law
rate of MPL against the exponential rate of SE, and MixExp residuals
against Exp residuals. Both changes lower mean NLS in every sample.
For MPL, the change of residual law gives a larger reduction than the change from SE to MPL under MixExp in
every sample except half-tick NVDA, and the MPL--SE difference is
smaller under MixExp than under Exp in all six samples.
Table~\ref{tab:prediction_nls_2x2} reports the one-tick paired
contrasts, and Appendix~\ref{App:prediction_ranges} gives all three
thresholds.

A rate-specification contrast compares MPL
with SE under the same residual law; a residual-law contrast
compares MixExp with Exp under the same rate specification. Parameter
counts are three under Exp and five under MixExp, and no penalty is
added to these out-of-sample scores.

\begin{table}[!tbp]
\centering\small
\caption{One-tick NLS contrasts on the 42 E-mini and 41 NVDA forecast
dates. MPL fixes $p=1.02$. Entries are
first-minus-second differences of date-balanced mean NLS with 95\%
paired date-block intervals; negative differences favor the first
specification. The label in parentheses gives the component held
fixed: the residual law in MPL--SE contrasts and the rate specification in
MixExp--Exp contrasts. Dates $<0$ counts dates favoring the first
specification. Appendix~\ref{App:prediction_ranges} gives all
three thresholds.}
\label{tab:prediction_nls_2x2}
\begin{tabular}{llrrr}
\toprule
Market & First minus second & $\Delta$NLS & 95\% interval & Dates $<0$ \\
\midrule
E-mini & MPL--SE (Exp) & $-0.0654$ & $[-0.0718, -0.0589]$ & 42/42 \\
 & MPL--SE (MixExp) & $-0.0103$ & $[-0.0145, -0.0064]$ & 39/42 \\
 & MixExp--Exp (SE) & $-0.0960$ & $[-0.1034, -0.0891]$ & 42/42 \\
 & MixExp--Exp (MPL) & $-0.0410$ & $[-0.0492, -0.0329]$ & 42/42 \\
NVDA & MPL--SE (Exp) & $-0.0496$ & $[-0.0569, -0.0429]$ & 41/41 \\
 & MPL--SE (MixExp) & $-0.0107$ & $[-0.0128, -0.0088]$ & 41/41 \\
 & MixExp--Exp (SE) & $-0.0600$ & $[-0.0670, -0.0532]$ & 41/41 \\
 & MixExp--Exp (MPL) & $-0.0211$ & $[-0.0249, -0.0174]$ & 41/41 \\
\bottomrule
\end{tabular}
\end{table}

At one tick, MPL lowers mean NLS under MixExp by 0.0103 for E-mini and
0.0107 for NVDA, on 39 of 42 and 41 of 41 dates, and both intervals
exclude zero. For MPL, changing Exp to MixExp lowers NLS by 0.0410 and
0.0211. The NVDA two-tick MPL--SE interval under Exp includes zero
(Appendix~\ref{App:prediction_ranges}).

Estimating $p$ under MixExp changes one-tick NLS by $-0.0008$ for
E-mini (95\% interval $[-0.0013,-0.0002]$) and $0.00004$ for NVDA
($[-0.0002,0.0003]$), small relative to the one-tick MPL--SE
contrasts. Across thresholds, the exception is half-tick E-mini, where a free exponent
lowers NLS by 0.0197 (Appendix~\ref{App:prediction_ranges}), about
half the MPL--SE contrast of 0.0423.

\subsubsection{Density forecasts across duration ranges}\label{Sect:pred_ranges}

We next ask in which duration ranges the forecast differences arise,
holding the residual law at MixExp; all three
specifications estimate five parameters. We use two measures: a
descriptive decomposition of the log score by realized duration, and
censored log scores that evaluate each range with cutoffs fixed
before forecasting.

The decomposition splits the log
score, $L_n=-\log f_n(\tau_n)$, whose date-balanced mean is NLS, into short, middle, and long
observations using each date's realized 10th and 90th duration
percentiles. Because the groups condition on the realized duration,
the decomposition shows where differences occur but is not a proper
scoring rule, under which the true predictive distribution attains
the best expected score \citep{GneitingRaftery2007}.
The censored log scores \citep{DiksPanchenkoVanDijk2011} provide
proper region-specific evaluation.
Their lower and upper cutoffs are the empirical 10th and 90th
percentiles of the preceding training window, held until the next
refit. They are common to all models and known before forecasting.
Every observation contributes to each censored score; observations
outside the target region contribute a probability rather than a
density. Appendix~\ref{App:prediction_ranges} gives the formulas.

\begin{table}[!tbp]
\centering\footnotesize
\setlength{\tabcolsep}{4pt}
\caption{One-step log-score differences, MPL minus the named
comparator. Full and censored scores (CS) are proper scores using
all origins. Lower/upper CS use the common training 10th/90th
percentile. Short, middle, and long scores are descriptive averages
within realized duration groups. Negative values favor MPL.}
\label{tab:prediction_density_ranges}
\begin{tabular}{lrlrrrrrr}
\toprule
 & & & \multicolumn{3}{c}{Proper scores} & \multicolumn{3}{c}{Realized-group averages} \\
\cmidrule(lr){4-6}\cmidrule(lr){7-9}
Market & Ticks & MPL vs. & Full & Lower CS & Upper CS & Short & Middle & Long \\
\midrule
E-mini & 0.5 & SE & $-$0.0423 & $-$0.0447 & 0.0018 & $-$0.6999 & 0.0161 & 0.1482 \\
 &  & ACD & $-$0.1444 & $-$0.1080 & 0.0022 & $-$1.7357 & 0.0221 & 0.1154 \\
E-mini & 1 & SE & $-$0.0103 & $-$0.0012 & $-$0.0034 & $-$0.0202 & $-$0.0064 & $-$0.0318 \\
 &  & ACD & $-$0.0647 & $-$0.0012 & 0.0095 & $-$0.3036 & $-$0.0432 & 0.0018 \\
E-mini & 2 & SE & $-$0.0039 & $-$0.0016 & $-$0.0022 & $-$0.0245 & 0.0016 & $-$0.0276 \\
 &  & ACD & 0.0061 & 0.0047 & 0.0002 & $-$0.0160 & $-$0.0023 & 0.0955 \\
NVDA & 0.5 & SE & $-$0.0543 & $-$0.0081 & $-$0.0109 & $-$0.2436 & $-$0.0336 & $-$0.0299 \\
 &  & ACD & $-$0.1800 & $-$0.0248 & $-$0.0109 & $-$0.8980 & $-$0.1057 & $-$0.0561 \\
NVDA & 1 & SE & $-$0.0107 & $-$0.0019 & $-$0.0042 & $-$0.0531 & 0.0008 & $-$0.0609 \\
 &  & ACD & $-$0.0339 & $-$0.0126 & 0.0076 & $-$0.3077 & $-$0.0155 & 0.0926 \\
NVDA & 2 & SE & $-$0.0055 & $-$0.0023 & $-$0.0035 & $-$0.0281 & 0.0017 & $-$0.0405 \\
 &  & ACD & 0.0133 & 0.0075 & 0.0046 & $-$0.0016 & $-$0.0063 & 0.1849 \\
\bottomrule
\end{tabular}
\end{table}

Against SE, MPL has lower censored scores in both tails in nearly every
sample. Against ACD, the advantage of MPL lies in the lower tail at
half and one tick, while ACD is better in the upper tail at one tick
and in the lower tail at two ticks.
\begin{itemize}
 \item Full score. MPL is lower than SE in all six samples. MPL is
 lower than ACD at half and one tick and higher at two ticks. All
 twelve intervals exclude zero (Appendix~\ref{App:prediction_ranges}).
 \item Lower tail. MPL is lower than SE in all six samples. Against
 ACD, MPL is lower for half-tick E-mini and for half- and one-tick
 NVDA, and higher at two ticks; for one-tick E-mini the interval
 includes zero.
 \item Upper tail. MPL is lower than SE in five samples; for half-tick
 E-mini the interval includes zero. Against ACD, MPL is higher for
 one-tick E-mini and for one- and two-tick NVDA and lower for half-tick
 NVDA; the intervals for half- and two-tick E-mini include zero.
 \item Realized groups. Against SE, MPL is lower in the short group
 and higher in the long group for half-tick E-mini, and lower in the
 long group at one and two ticks. Against ACD, MPL is lower in the
 short group at all three thresholds, with much smaller differences
 at two ticks.
\end{itemize}

The large half-tick short-group contrasts concern
sub-millisecond durations and can be sensitive to timestamp
resolution and the consolidation of updates sharing a timestamp.

\subsubsection{Quantile forecasts and tail frequencies}\label{Sect:pred_quantiles}

We next evaluate forecast quantiles $\widehat q_n(u)=F_n^{-1}(u)$, $0<u<1$,
using pinball loss,
\[
\rho_u(\tau_n-\widehat q_n(u))
=(u-\mathbbm{1}_{\{\tau_n<\widehat q_n(u)\}})(\tau_n-\widehat q_n(u)).
\]
Let $\mathrm{PB}_m(u)$ denote the date-balanced mean pinball loss of
model $m$ at quantile level $u$. We also
report the daily absolute difference between the frequency of
$\tau_n>\widehat q_n(0.99)$ and its reference probability of 1\%.
This exceedance frequency assesses how often durations exceed the forecast quantile, whereas
pinball loss also depends on the sizes of forecast errors.

\begin{table}[!tbp]
\centering\small
\caption{Pinball-loss and upper-1\% exceedance comparisons of MPL with
SE and ACD. Each row compares MPL with the named model. The
pinball-loss columns are relative differences
$100(\mathrm{PB}_{\mathrm{MPL}}(u)/\mathrm{PB}_{\mathrm{comparator}}(u)-1)$,
in percent. The last column is the difference in mean
daily absolute upper-1\% exceedance error, in percentage points.
Negative values favor MPL. Appendix~\ref{App:prediction_ranges}
gives the 10th and 50th percentiles.}
\label{tab:prediction_quantiles}
\begin{tabular}{lrlrrr}
\toprule
Market & Ticks & MPL vs. & $u=0.90$ & $u=0.99$ & Upper-1\% error \\
\midrule
E-mini & 0.5 & SE & 8.50 & 2.68 & $-$0.441 \\
 &  & ACD & 11.83 & 3.25 & $-$0.904 \\
E-mini & 1 & SE & 0.49 & $-$3.24 & $-$0.347 \\
 &  & ACD & 6.19 & 4.69 & $-$0.638 \\
E-mini & 2 & SE & $-$0.96 & $-$1.76 & $-$0.033 \\
 &  & ACD & 0.71 & $-$0.88 & $-$0.266 \\
NVDA & 0.5 & SE & 1.38 & $-$4.51 & $-$0.568 \\
 &  & ACD & 4.56 & $-$4.20 & $-$0.969 \\
NVDA & 1 & SE & 0.43 & $-$2.09 & $-$0.244 \\
 &  & ACD & 5.63 & 6.93 & $-$0.457 \\
NVDA & 2 & SE & $-$1.27 & $-$1.80 & $-$0.058 \\
 &  & ACD & 1.47 & 0.29 & $-$0.408 \\
\bottomrule
\end{tabular}
\end{table}

MPL has the exceedance frequency closest to its target in all six
samples, but this does not give it lower pinball loss: ACD has lower
90th-percentile loss than MPL in every sample.
\begin{itemize}
 \item $u=0.10$ and $u=0.50$. Relative differences are below 0.6\% and
 1.5\% in absolute value in every comparison
 (Appendix~\ref{App:prediction_ranges}), so the large short-group
 log-score gains do not produce large lower-quantile gains.
 \item $u=0.90$. ACD is lower than MPL in all six samples; for two-tick
 E-mini the interval includes zero. SE is lower than MPL at half and
 one tick, with an interval including zero for one-tick NVDA, and MPL
 is lower than SE at two ticks.
 \item $u=0.99$. MPL is lower than SE in five samples and higher for
 half-tick E-mini. Against ACD, MPL is higher for half- and one-tick
 E-mini and one-tick NVDA and lower for half-tick NVDA; the two-tick
 intervals include zero.
 \item Upper-1\% exceedance. MPL has the smallest absolute error in all
 six samples. Mean exceedance rates remain above 1\% for all three
 models: 1.03--1.75\% for MPL, 1.07--2.12\% for SE, and 1.53--2.58\%
 for ACD (Appendix~\ref{App:prediction_ranges}). All three therefore
 underpredict the frequency of durations beyond their conditional
 99th percentiles.
\end{itemize}
\FloatBarrier
\endgroup
\section{Conclusion}\label{Sect:conc}

We constructed a Hawkes-inspired point-process model with power-law
rate decay and a flexible residual distribution. Matching the Hawkes jumps in
level and slope gives a two-dimensional state at event times.
Under the residual density assumption, a finite residual mean
$\mu_\varepsilon<\infty$, and the stability condition
$\xi\gamma<\mu_\varepsilon p^2/(p+1)$, the embedded chain is positive
Harris recurrent and converges geometrically in total variation to its
unique invariant probability measure from every initial state. For
unit exponential residuals, this condition covers all matched
subcritical Hawkes parameter values.

In E-mini and NVDA midquote durations, the fitted parameters of the Markov power-law (MPL) model in the
one-tick Exp and MixExp comparison lie inside the sufficient stability region.
In-sample, replacing unit exponential residuals by a two-component
exponential mixture lowers the PIT distance for every model and
threshold, and every flexible residual law lowers BIC relative to
Exp. At half and one tick, MPL and SE have smaller PIT distances than
ACD under every residual law, while ACD leaves the least residual
autocorrelation in nearly every specification; at two ticks, the PIT
ordering depends on the market and residual law.
Fixing $p=1.02$ costs little in the one-tick samples, although a
free exponent improves half-tick E-mini forecasts.

Out of sample, with the common mixture law, MPL has a lower mean
negative log predictive score (NLS) than SE at all three thresholds
in both markets. For MPL, changing the residual law from Exp to
MixExp lowers NLS more than changing the rate specification from SE
to MPL under MixExp, except in half-tick NVDA. Relative to ACD, MPL has lower NLS
at half and one tick, with its largest per-observation gains among
short durations, while ACD has lower NLS at two ticks. ACD has the
lowest mean CRPS and RMSE in every sample, within 3\% of those of
MPL, and lower 90th-percentile loss than MPL. MPL comes closest to
the target upper-1\% exceedance frequency, although all three models
exceed it. No specification is best on all of PIT distance, residual
autocorrelation, NLS, CRPS, and quantile loss.

\section*{Funding}
This work was supported by the National Research Foundation of Korea (NRF) grant funded by the Korea government (MSIT) (No. RS-2026-25469087).

\appendix
\setcounter{theorem}{0}
\setcounter{definition}{0}
\renewcommand{\thetheorem}{A.\arabic{theorem}}
\renewcommand{\thedefinition}{A.\arabic{definition}}
\def\theHtheorem{A.\arabic{theorem}}
\def\theHdefinition{A.\arabic{definition}}
\section{Markov chain background}\label{App:MC}

We use the terminology of \citet{meyn2009markov}. Let $P$ be a
transition kernel on a Borel state space $\mathcal S$, and write
$\mathcal B(\mathcal S)$ for its Borel sets.

\begin{definition}\label{Def:MC_terms}
\begin{enumerate}
 \item[(i)] For a nonzero measure $\psi$ on $\mathcal B(\mathcal S)$,
 the chain is $\psi$-irreducible if, for every $s\in\mathcal S$ and
 every $B$ with $\psi(B)>0$, there is $m\ge1$ with $P^m(s,B)>0$.
 \item[(ii)] A set $C\in\mathcal B(\mathcal S)$ is small if there are
 $m\ge1$ and a nonzero measure $\nu$ with $P^m(s,B)\ge\nu(B)$ for all
 $s\in C$ and $B\in\mathcal B(\mathcal S)$. It is petite if the same
 bound holds with $P^m$ replaced by $\sum_{m\ge0}a_mP^m$ for some
 probability distribution $(a_m)_{m\ge0}$. Every small set is petite.
 \item[(iii)] A $\psi$-irreducible chain is aperiodic if there are no
 $d\ge2$ and disjoint sets $E_1,\dots,E_d\in\mathcal B(\mathcal S)$
 such that $P(s,E_{i+1})=1$ for $s\in E_i$, $i=1,\dots,d$, with
 $E_{d+1}=E_1$, and the complement of $E_1\cup\cdots\cup E_d$ is
 null for a maximal irreducibility measure.
 \item[(iv)] A $\psi$-irreducible chain is Harris recurrent if, from
 every starting state, it enters every set $B$ with $\psi(B)>0$
 infinitely often almost surely. It is positive Harris recurrent if,
 in addition, it has an invariant probability measure.
 \item[(v)] For a measurable function $V\ge1$ and a finite signed
 measure $\nu$,
 \[
  \|\nu\|_V=\sup_{|g|\le V}\Bigl|\int g\,\D\nu\Bigr|,
 \]
 where the supremum runs over bounded measurable $g$. The case $V\equiv1$ is
 the total variation norm $\|\nu\|_{\mathrm{TV}}$.
\end{enumerate}
\end{definition}

The first result is \citet[Theorem~9.1.8]{meyn2009markov}. A
function $V:\mathcal S\to[0,\infty)$ is unbounded off petite sets if
every sublevel set $\{s:V(s)\le n\}$, $n\ge1$, is petite.

\begin{theorem}\label{Thm:MT_harris}
Suppose the chain is $\psi$-irreducible. If there are a petite set $C$
and a function $V$ that is unbounded off petite sets such that
\[
 PV(s)\le V(s),\qquad s\notin C,
\]
then the chain is Harris recurrent.
\end{theorem}

The second result is the part of the geometric ergodic theorem,
\citet[Theorem~15.0.1]{meyn2009markov}, that derives geometric
convergence from a drift condition.

\begin{theorem}\label{Thm:MT_geometric}
Suppose the chain is $\psi$-irreducible and aperiodic. Suppose also
that there are a petite set $C$, constants $b<\infty$ and $\delta>0$,
and a measurable function $V\ge1$, finite at some point, such that
\[
 PV(s)-V(s)\le-\delta V(s)+b\,\mathbbm{1}_C(s),\qquad s\in\mathcal S.
\]
Then the chain has an invariant probability measure $\pi$, the set
$\{s:V(s)<\infty\}$ is absorbing and full, and there are constants
$r>1$ and $R<\infty$ such that, for every $s$ with $V(s)<\infty$,
\[
 \sum_{n\ge0}r^n\|P^n(s,\cdot)-\pi\|_V\le R\,V(s).
\]
\end{theorem}

\section{Residual distributions}\label{App:residuals}

All residual laws are normalized to mean one. Their densities on
$z>0$ and parameter ranges are as follows.
\begin{itemize}
 \item Exp: $f(z)=\e^{-z}$.
 \item MixExp:
 \[
  f(z)=\frac{w}{s_1}\e^{-z/s_1}+\frac{1-w}{s_2}\e^{-z/s_2},\qquad
  s_1=\frac{1}{w+(1-w)r},\quad s_2=rs_1,
 \]
 with $w\in[0.005,0.995]$ and $r\in[1.00001,10^6]$. The component
 means are $s_1$ and $s_2$.
 \item Burr type XII:
 \[
  f(z)=\frac{ab}{s}\Bigl(\frac zs\Bigr)^{b-1}
       \Bigl[1+\Bigl(\frac zs\Bigr)^b\Bigr]^{-a-1},\qquad
  s=\frac{\Gamma(a)}{\Gamma(1+1/b)\Gamma(a-1/b)},
 \]
 with $a>1/b$ and $(b,a-1/b)\in[0.03,30]\times[0.0015,10^5]$.
 \item Generalized gamma:
 \[
  f(z)=\frac{b}{s^d\Gamma(d/b)}z^{d-1}\e^{-(z/s)^b},\qquad
  s=\frac{\Gamma(d/b)}{\Gamma((d+1)/b)},
 \]
 with $(d,b)\in[0.01,30]\times[0.03,30]$.
\end{itemize}

\section{Supporting MPL estimation and fit diagnostics}\label{App:basic_support}
\setcounter{table}{0}
\renewcommand{\thetable}{\thesection.\arabic{table}}

We reconstruct the best bid and ask from order-book messages,
selecting the E-mini maturity with the largest volume in the preceding
session and using Nasdaq quotes for NVDA. Within each day, we retain
the longest valid segment and form no durations across missing,
locked, or crossed books. Updates sharing a timestamp are consolidated
to the final quote. Days with fewer than 5,010 contiguous durations
are excluded.

Parameters are estimated on rolling windows of 5,000 durations,
refitted every 100 observations. In-sample diagnostics use
nonoverlapping windows starting at the beginning of the day and every
5,000 durations thereafter. Each window is filtered with its fitted
parameters and estimation initial state, and residuals are pooled
within the date. On June 2, the illustration date fixed in advance
(Section~\ref{Sect:fit_thresholds}), the one-tick samples contain
85,000 and 60,000 durations. ACFs join consecutive windows within a
date; omitting lag pairs across window boundaries changes ACF RMS by
at most 0.0016. Out-of-sample forecasts use each fit for the next
block of up to 100 durations, with states updated after each
observation. The first 5,000 durations and final nine forecast origins
are excluded; the latter restriction retains the evaluation sample
originally constructed for forecasts of up to ten future durations.

Optimization uses multiple starts every twentieth refit. MixExp fits
are reoptimized when their likelihood falls below that of Exp,
obtained as the component mean ratio approaches one. Free-$p$ fits are
initialized at the fixed-$p$ estimates.
\FloatBarrier
\section{Supporting results for the extended goodness-of-fit comparison}\label{App:extended_gof}
\setcounter{table}{0}
\renewcommand{\thetable}{\thesection.\arabic{table}}

\subsection{Estimation}

Appendix~\ref{App:residuals} defines the residual laws. For MPL and SE, rate parameters are optimized after scaling each
training window to mean one, over
$(\mu,\xi,\gamma)\in[10^{-5},50]\times[10^{-6},10^4]\times[10^{-5},10^4]$.
ACD uses a positive intercept and a softmax parameterization of
$(a_1,b_1,1-a_1-b_1)$. Fits use warm starts, with multiple starts
every twentieth refit, and generalized gamma and MixExp fits are
compared with their exponential special or limiting case. Likelihood
gradients are analytic. Gamma upper-tail probabilities are evaluated
in log form when they underflow, and marginal quantiles are computed
by root finding on the average conditional CDF.

Burr and generalized gamma often reach the shape search limits
(Table~\ref{tab:extended_boundaries}), so their fitted shapes and
rankings depend on the stated ranges; MixExp rarely does. The
stability restriction is not imposed. Four of 953,616 MPL parameter
vectors have $R\ge1$, all under Burr for two-tick NVDA on May 21, with
largest value 1.0127; the stability theorem does not cover these
fits. All fits in the one-tick Exp/MixExp comparison satisfy $R<1$.

\begin{table}[!tbp]
\centering\small
\caption{Mean fraction of rolling fits within 0.1\% of a residual-shape
search bound, in percent. Each retained market--threshold--date
receives equal weight.}
\label{tab:extended_boundaries}
\begin{tabular}{lrrr}
\toprule
Residual & MPL & SE & ACD \\
\midrule
Burr & 40.1 & 47.5 & 58.4 \\
MixExp & 2.0 & 1.8 & 0.6 \\
Gen. gamma & 13.6 & 10.8 & 29.8 \\
\bottomrule
\end{tabular}
\end{table}

\subsection{Burr and generalized gamma residuals}

Table~\ref{tab:extended_flex} completes Table~\ref{tab:extended_main}
for the remaining two residual laws.

\begin{table}[!tbp]
\centering\footnotesize
\setlength{\tabcolsep}{3.5pt}
\caption{In-sample fit under Burr and generalized gamma residuals:
date-balanced mean PIT Kolmogorov distance, exponential-residual ACF
RMS at lags 1--50, and within-window $\Delta$BIC over the selected
nonoverlapping training windows. Lower is better. Bold entries are
the lowest of the three models for each statistic.}
\label{tab:extended_flex}
\begin{tabular}{lrl rrr rrr rrr}
\toprule
 & & & \multicolumn{3}{c}{PIT distance} & \multicolumn{3}{c}{ACF RMS}
 & \multicolumn{3}{c}{$\Delta$BIC} \\
\cmidrule(lr){4-6}\cmidrule(lr){7-9}\cmidrule(lr){10-12}
Market & Ticks & Residual & MPL & SE & ACD & MPL & SE & ACD & MPL & SE & ACD \\
\midrule
E-mini & 0.5 & Burr & 0.0539 & \textbf{0.0520} & 0.1367 & 0.1110 & 0.1108 & \textbf{0.0972} & 144.9 & \textbf{63.5} & 1357.5 \\
 &  & Gen. gamma & \textbf{0.0535} & 0.0562 & 0.1193 & 0.1106 & 0.1072 & \textbf{0.1003} & 138.8 & \textbf{79.3} & 965.4 \\
E-mini & 1 & Burr & 0.0631 & \textbf{0.0412} & 0.1028 & 0.0384 & 0.0388 & \textbf{0.0087} & \textbf{362.4} & 426.2 & 1576.2 \\
 &  & Gen. gamma & 0.0610 & \textbf{0.0391} & 0.0960 & 0.0391 & 0.0396 & \textbf{0.0087} & \textbf{264.2} & 319.6 & 1453.4 \\
E-mini & 2 & Burr & 0.0269 & \textbf{0.0263} & 0.0293 & 0.0223 & 0.0229 & \textbf{0.0120} & 243.0 & 274.6 & \textbf{62.8} \\
 &  & Gen. gamma & 0.0273 & \textbf{0.0265} & 0.0296 & 0.0204 & 0.0215 & \textbf{0.0118} & 150.9 & 180.6 & \textbf{49.6} \\
NVDA & 0.5 & Burr & \textbf{0.0318} & 0.0355 & 0.1160 & 0.1708 & 0.1757 & \textbf{0.1613} & \textbf{104.9} & 535.8 & 2196.0 \\
 &  & Gen. gamma & 0.0396 & \textbf{0.0347} & 0.1009 & \textbf{0.1674} & 0.1777 & 0.1701 & \textbf{65.0} & 328.2 & 1673.3 \\
NVDA & 1 & Burr & \textbf{0.0134} & 0.0158 & 0.0686 & 0.0624 & 0.0645 & \textbf{0.0095} & \textbf{127.7} & 198.4 & 708.5 \\
 &  & Gen. gamma & \textbf{0.0148} & 0.0162 & 0.0418 & 0.0527 & 0.0528 & \textbf{0.0098} & \textbf{96.1} & 162.7 & 472.7 \\
NVDA & 2 & Burr & 0.0403 & \textbf{0.0392} & 0.0437 & 0.0329 & 0.0322 & \textbf{0.0123} & 399.3 & 424.7 & \textbf{164.2} \\
 &  & Gen. gamma & 0.0267 & \textbf{0.0247} & 0.0302 & 0.0218 & 0.0234 & \textbf{0.0120} & 178.9 & 205.8 & \textbf{69.0} \\
\bottomrule
\end{tabular}
\end{table}

\subsection{Information criterion and duration quantiles}

In each selected window, $\mathrm{BIC}=-2\ell+k\log(5000)$, where
$\ell$ is the maximized conditional log-likelihood on that window,
with $k=3$ under Exp and $k=5$ under the flexible laws. Using all rolling
windows instead changes the values but not the ordering of the three
models for any market, threshold, or residual law.

For each date, the fitted marginal CDF averages the fitted conditional
CDFs over the pre-duration states of the $n_d$ diagnostic observations
of date $d$, $\widetilde F_d(t)=n_d^{-1}\sum_n F_n(t)$. It describes the observed
states, not the stationary duration law. The 99th-percentile error
is $|\widetilde F_d^{-1}(0.99)/q_d^{\mathrm{emp}}(0.99)-1|$, where
$q_d^{\mathrm{emp}}(0.99)$ is the empirical 99th percentile of date $d$.

\begin{table}[!tbp]
\centering\small
\caption{Marginal 99th-percentile error and residual variance. The
first block is the mean absolute relative error of the marginal
duration 99th percentile, in percent, computed by date before
averaging; bold entries are the lowest of the three models. The
second block is the date-balanced mean variance of the exponential
residuals, whose reference value is one.}
\label{tab:extended_q99}
\begin{tabular}{lrl rrr rrr}
\toprule
 & & & \multicolumn{3}{c}{99th-percentile error (\%)} & \multicolumn{3}{c}{$\Var(Z)$} \\
\cmidrule(lr){4-6}\cmidrule(lr){7-9}
Market & Ticks & Residual & MPL & SE & ACD & MPL & SE & ACD \\
\midrule
E-mini & 0.5 & Exp & \textbf{8.3} & 24.0 & 59.1 & 1.163 & 1.781 & 8.900 \\
 &  & Burr & \textbf{15.5} & 18.4 & 97.1 & 0.983 & 0.977 & 1.108 \\
 &  & MixExp & \textbf{7.6} & 16.6 & 27.4 & 1.209 & 1.381 & 1.627 \\
 &  & Gen. gamma & 15.6 & \textbf{6.5} & 184.8 & 0.983 & 1.015 & 1.015 \\
E-mini & 1 & Exp & \textbf{3.5} & 21.8 & 36.2 & 1.099 & 1.556 & 3.281 \\
 &  & Burr & \textbf{7.4} & 43.8 & 73.2 & 0.977 & 0.926 & 0.874 \\
 &  & MixExp & \textbf{4.6} & 12.7 & 17.2 & 1.093 & 1.204 & 1.426 \\
 &  & Gen. gamma & \textbf{9.9} & 20.5 & 45.1 & 0.953 & 0.945 & 1.000 \\
E-mini & 2 & Exp & 8.6 & 9.1 & \textbf{4.8} & 0.929 & 0.949 & 1.296 \\
 &  & Burr & \textbf{10.2} & 11.7 & 10.7 & 0.913 & 0.912 & 1.041 \\
 &  & MixExp & 7.2 & 7.2 & \textbf{5.5} & 0.921 & 0.938 & 1.170 \\
 &  & Gen. gamma & \textbf{3.3} & 3.8 & 8.3 & 1.035 & 1.030 & 1.091 \\
NVDA & 0.5 & Exp & \textbf{3.4} & 23.1 & 50.1 & 1.204 & 1.903 & 8.397 \\
 &  & Burr & 76.8 & \textbf{39.7} & 82.6 & 0.874 & 0.954 & 1.089 \\
 &  & MixExp & \textbf{4.4} & 14.7 & 25.2 & 1.114 & 1.389 & 1.720 \\
 &  & Gen. gamma & \textbf{6.9} & 40.9 & 206.1 & 1.006 & 0.925 & 1.029 \\
NVDA & 1 & Exp & \textbf{9.1} & 18.9 & 17.9 & 1.141 & 1.486 & 2.818 \\
 &  & Burr & 2.2 & \textbf{1.8} & 46.7 & 0.990 & 0.980 & 0.873 \\
 &  & MixExp & \textbf{4.5} & 5.6 & 9.4 & 1.078 & 1.154 & 1.455 \\
 &  & Gen. gamma & 1.5 & \textbf{0.9} & 10.5 & 1.007 & 0.998 & 1.055 \\
NVDA & 2 & Exp & 7.7 & 8.3 & \textbf{3.5} & 1.067 & 1.150 & 1.648 \\
 &  & Burr & 18.5 & 19.9 & \textbf{15.9} & 0.837 & 0.835 & 0.985 \\
 &  & MixExp & 5.0 & 5.0 & \textbf{3.4} & 0.991 & 1.010 & 1.299 \\
 &  & Gen. gamma & \textbf{2.7} & 3.2 & 9.7 & 1.037 & 1.032 & 1.098 \\
\bottomrule
\end{tabular}
\end{table}

\subsection{Residual moments and paired comparisons}

For MPL and SE, the in-sample residual mean is within 0.007 of one in
every case, so Table~\ref{tab:extended_q99} reports only the
variance. Using all available dates changes only the E-mini two-tick
sample and leaves the specification with the lowest PIT distance,
ACF RMS, $\Delta$BIC, and 99th-percentile error unchanged in every
sample.

\begin{table}[!tbp]
\centering\footnotesize
\setlength{\tabcolsep}{4pt}
\caption{Paired daily comparisons under MixExp: differences in PIT
distance and residual ACF RMS, MPL minus the named model, with 95\%
date-block intervals. Negative values favor MPL.}
\begin{tabular}{lrl ll}
\toprule
Market & Ticks & MPL vs. & PIT distance & ACF RMS \\
\midrule
E-mini & 0.5 & SE & 0.0187 [0.0143, 0.0231] & 0.0174 [0.0147, 0.0203] \\
 &  & ACD & $-$0.0301 [$-$0.0358, $-$0.0242] & 0.0337 [0.0318, 0.0357] \\
E-mini & 1 & SE & 0.0025 [0.0013, 0.0038] & $-$0.0013 [$-$0.0021, $-$0.0003] \\
 &  & ACD & $-$0.0303 [$-$0.0342, $-$0.0265] & 0.0215 [0.0202, 0.0231] \\
E-mini & 2 & SE & 0.0036 [0.0032, 0.0040] & $-$0.0003 [$-$0.0011, 0.0005] \\
 &  & ACD & 0.0063 [0.0012, 0.0105] & 0.0101 [0.0051, 0.0154] \\
NVDA & 0.5 & SE & $-$0.0161 [$-$0.0188, $-$0.0133] & 0.0183 [0.0170, 0.0197] \\
 &  & ACD & $-$0.0948 [$-$0.0990, $-$0.0902] & 0.0433 [0.0402, 0.0461] \\
NVDA & 1 & SE & $-$0.0062 [$-$0.0081, $-$0.0048] & 0.0019 [0.0009, 0.0029] \\
 &  & ACD & $-$0.0402 [$-$0.0439, $-$0.0365] & 0.0306 [0.0283, 0.0325] \\
NVDA & 2 & SE & 0.0007 [$-$0.0002, 0.0016] & $-$0.0003 [$-$0.0008, 0.0002] \\
 &  & ACD & $-$0.0156 [$-$0.0183, $-$0.0128] & 0.0100 [0.0065, 0.0136] \\
\bottomrule
\end{tabular}
\end{table}
\FloatBarrier
\section{Supporting prediction comparisons}\label{App:prediction_ranges}
\setcounter{table}{0}
\renewcommand{\thetable}{\thesection.\arabic{table}}

The forecasts use the same rolling fits as the goodness-of-fit
analysis. Forecast quantiles are obtained by inverting each
conditional CDF numerically.

\subsection{Full predictive-density contrasts}

Table~\ref{tab:prediction_nls_all_thresholds} extends the one-tick
contrasts to all three price thresholds, using the same dates and
origins within each market and threshold. Rate contrasts compare
equal parameter counts; residual-law contrasts compare three
parameters with five, and free-$p$ contrasts six with five. All
rate and residual parameters are reestimated for each
specification.

\begin{table}[!tbp]
\centering\footnotesize
\setlength{\tabcolsep}{4pt}
\caption{Full-sample NLS contrasts on the forecast dates of
Section~\ref{Sect:prediction_ranges}. Differences are first minus second, multiplied by
$10^3$; negative values favor the first specification. The label in
parentheses gives the component held fixed: the residual law in
MPL--SE and free-$p$ contrasts and the rate specification in MixExp--Exp
contrasts. Dates $<0$ counts negative daily differences. MPL fixes $p=1.02$ except in the
free-$p$ rows.}
\label{tab:prediction_nls_all_thresholds}
\begin{tabular}{lr lrrr}
\toprule
Market & Ticks & First minus second & $10^3\,\Delta$NLS & 95\% interval & Dates $<0$ \\
\midrule
E-mini & 0.5 & MPL--SE (Exp) & $-133.99$ & $[-151.81, -118.30]$ & 43/43 \\
 &  & MPL--SE (MixExp) & $-42.31$ & $[-50.86, -34.80]$ & 42/43 \\
 &  & MixExp--Exp (SE) & $-154.25$ & $[-178.06, -133.53]$ & 43/43 \\
 &  & MixExp--Exp (MPL) & $-62.58$ & $[-75.10, -52.00]$ & 43/43 \\
 &  & Free $p$--fixed $p$ (MixExp) & $-19.67$ & $[-23.56, -15.98]$ & 43/43 \\
E-mini & 1 & MPL--SE (Exp) & $-65.42$ & $[-71.78, -58.94]$ & 42/42 \\
 &  & MPL--SE (MixExp) & $-10.34$ & $[-14.52, -6.38]$ & 39/42 \\
 &  & MixExp--Exp (SE) & $-96.03$ & $[-103.44, -89.12]$ & 42/42 \\
 &  & MixExp--Exp (MPL) & $-40.95$ & $[-49.16, -32.88]$ & 42/42 \\
 &  & Free $p$--fixed $p$ (MixExp) & $-0.76$ & $[-1.34, -0.23]$ & 29/42 \\
E-mini & 2 & MPL--SE (Exp) & $-4.34$ & $[-6.86, -1.61]$ & 28/29 \\
 &  & MPL--SE (MixExp) & $-3.94$ & $[-5.15, -2.76]$ & 28/29 \\
 &  & MixExp--Exp (SE) & $-20.10$ & $[-27.55, -12.34]$ & 29/29 \\
 &  & MixExp--Exp (MPL) & $-19.70$ & $[-27.98, -11.46]$ & 27/29 \\
 &  & Free $p$--fixed $p$ (MixExp) & $0.03$ & $[-0.03, 0.09]$ & 20/29 \\
NVDA & 0.5 & MPL--SE (Exp) & $-159.38$ & $[-181.84, -142.15]$ & 41/41 \\
 &  & MPL--SE (MixExp) & $-54.26$ & $[-64.62, -45.82]$ & 41/41 \\
 &  & MixExp--Exp (SE) & $-145.13$ & $[-153.89, -137.67]$ & 41/41 \\
 &  & MixExp--Exp (MPL) & $-40.02$ & $[-44.05, -35.40]$ & 41/41 \\
 &  & Free $p$--fixed $p$ (MixExp) & $-1.03$ & $[-1.44, -0.66]$ & 39/41 \\
NVDA & 1 & MPL--SE (Exp) & $-49.62$ & $[-56.87, -42.89]$ & 41/41 \\
 &  & MPL--SE (MixExp) & $-10.72$ & $[-12.84, -8.75]$ & 41/41 \\
 &  & MixExp--Exp (SE) & $-60.03$ & $[-66.95, -53.22]$ & 41/41 \\
 &  & MixExp--Exp (MPL) & $-21.13$ & $[-24.91, -17.42]$ & 41/41 \\
 &  & Free $p$--fixed $p$ (MixExp) & $0.04$ & $[-0.23, 0.26]$ & 15/41 \\
NVDA & 2 & MPL--SE (Exp) & $-8.88$ & $[-15.89, 0.44]$ & 38/41 \\
 &  & MPL--SE (MixExp) & $-5.52$ & $[-6.07, -4.94]$ & 41/41 \\
 &  & MixExp--Exp (SE) & $-62.57$ & $[-76.13, -50.06]$ & 41/41 \\
 &  & MixExp--Exp (MPL) & $-59.21$ & $[-74.50, -45.91]$ & 41/41 \\
 &  & Free $p$--fixed $p$ (MixExp) & $0.00$ & $[-0.01, 0.02]$ & 29/41 \\
\bottomrule
\end{tabular}
\end{table}

\subsection{Region definitions and scores}

Short, middle, and long observations satisfy
$\tau_n\le q_d^{\mathrm{emp}}(0.10)$,
$q_d^{\mathrm{emp}}(0.10)<\tau_n<q_d^{\mathrm{emp}}(0.90)$, and
$\tau_n\ge q_d^{\mathrm{emp}}(0.90)$, where
$q_d^{\mathrm{emp}}(u)$ is the realized $u$-quantile of date $d$. Each group score is
a mean over its own observations, and weighting the three by their
observation fractions recovers the full score. The groups are not
known at the forecast origin.

For censored scores, let $t_n^L$ and $t_n^U$ be the 10th and 90th
empirical percentiles of the training observations at the most
recent refit. With $L_n=-\log f_n(\tau_n)$, we use
\begin{align*}
 L^L_n&=\mathbbm{1}_{\{\tau_n\le t_n^L\}}L_n
   -\mathbbm{1}_{\{\tau_n>t_n^L\}}\log[1-F_n(t_n^L)],\\
 L^U_n&=\mathbbm{1}_{\{\tau_n\ge t_n^U\}}L_n
   -\mathbbm{1}_{\{\tau_n<t_n^U\}}\log F_n(t_n^U).
\end{align*}
These censored scores \citep{DiksPanchenkoVanDijk2011} keep the
forecast probability of falling outside the region of interest. The
cutoffs are common to all models and use only data preceding the
forecast.

\subsection{Paired uncertainty and exceedance frequencies}

All differences pair the same dates and forecast origins. Intervals
use 10,000 circular date-block bootstrap draws with blocks of five
consecutive dates and are not adjusted for multiple comparisons.
Negative differences favor MPL.

\begin{table}[!tbp]
\centering\footnotesize
\setlength{\tabcolsep}{4pt}
\caption{Mean log-score differences and 95\% date-block intervals.
Differences are MPL minus the comparator. Censored scores use
all origins and common training cutoffs.}
\begin{tabular}{lrllll}
\toprule
Market & Ticks & MPL vs. & Full & Lower CS & Upper CS \\
\midrule
E-mini & 0.5 & SE & \shortstack{$-$0.0423\\{[$-$0.0509, $-$0.0348]}} & \shortstack{$-$0.0447\\{[$-$0.0549, $-$0.0358]}} & \shortstack{0.0018\\{[$-$0.0030, 0.0060]}} \\
 &  & ACD & \shortstack{$-$0.1444\\{[$-$0.1720, $-$0.1214]}} & \shortstack{$-$0.1080\\{[$-$0.1372, $-$0.0813]}} & \shortstack{0.0022\\{[$-$0.0013, 0.0058]}} \\
E-mini & 1 & SE & \shortstack{$-$0.0103\\{[$-$0.0145, $-$0.0064]}} & \shortstack{$-$0.0012\\{[$-$0.0015, $-$0.0010]}} & \shortstack{$-$0.0034\\{[$-$0.0038, $-$0.0030]}} \\
 &  & ACD & \shortstack{$-$0.0647\\{[$-$0.0731, $-$0.0571]}} & \shortstack{$-$0.0012\\{[$-$0.0030, 0.0006]}} & \shortstack{0.0095\\{[0.0076, 0.0111]}} \\
E-mini & 2 & SE & \shortstack{$-$0.0039\\{[$-$0.0052, $-$0.0028]}} & \shortstack{$-$0.0016\\{[$-$0.0020, $-$0.0010]}} & \shortstack{$-$0.0022\\{[$-$0.0032, $-$0.0014]}} \\
 &  & ACD & \shortstack{0.0061\\{[0.0017, 0.0110]}} & \shortstack{0.0047\\{[0.0027, 0.0067]}} & \shortstack{0.0002\\{[$-$0.0017, 0.0020]}} \\
NVDA & 0.5 & SE & \shortstack{$-$0.0543\\{[$-$0.0646, $-$0.0458]}} & \shortstack{$-$0.0081\\{[$-$0.0091, $-$0.0069]}} & \shortstack{$-$0.0109\\{[$-$0.0162, $-$0.0068]}} \\
 &  & ACD & \shortstack{$-$0.1800\\{[$-$0.2093, $-$0.1573]}} & \shortstack{$-$0.0248\\{[$-$0.0282, $-$0.0217]}} & \shortstack{$-$0.0109\\{[$-$0.0212, $-$0.0031]}} \\
NVDA & 1 & SE & \shortstack{$-$0.0107\\{[$-$0.0128, $-$0.0088]}} & \shortstack{$-$0.0019\\{[$-$0.0028, $-$0.0012]}} & \shortstack{$-$0.0042\\{[$-$0.0050, $-$0.0035]}} \\
 &  & ACD & \shortstack{$-$0.0339\\{[$-$0.0412, $-$0.0276]}} & \shortstack{$-$0.0126\\{[$-$0.0159, $-$0.0100]}} & \shortstack{0.0076\\{[0.0064, 0.0090]}} \\
NVDA & 2 & SE & \shortstack{$-$0.0055\\{[$-$0.0061, $-$0.0049]}} & \shortstack{$-$0.0023\\{[$-$0.0028, $-$0.0019]}} & \shortstack{$-$0.0035\\{[$-$0.0039, $-$0.0032]}} \\
 &  & ACD & \shortstack{0.0133\\{[0.0077, 0.0194]}} & \shortstack{0.0075\\{[0.0043, 0.0110]}} & \shortstack{0.0046\\{[0.0017, 0.0080]}} \\
\bottomrule
\end{tabular}
\end{table}

\begin{table}[!tbp]
\centering\footnotesize
\setlength{\tabcolsep}{4pt}
\caption{Relative pinball-loss differences
$100(\mathrm{PB}_{\mathrm{MPL}}(u)/\mathrm{PB}_{\mathrm{comparator}}(u)-1)$,
in percent, and
upper-1\% exceedance-error differences, in percentage points.
Intervals are 95\% date-block intervals; the 10th- and
50th-percentile columns give point values only. The exceedance
error is the absolute daily exceedance-frequency deviation from 1\%.
Negative values favor MPL.}
\begin{tabular}{lrlrrlll}
\toprule
Market & Ticks & MPL vs. & $u=0.10$ & $u=0.50$ & $u=0.90$ loss (\%) & $u=0.99$ loss (\%) & Upper-1\% error (pp) \\
\midrule
E-mini & 0.5 & SE & $-$0.005 & $-$0.20 & \shortstack{8.50\\{[6.72, 10.18]}} & \shortstack{2.68\\{[1.48, 3.79]}} & \shortstack{$-$0.441\\{[$-$0.565, $-$0.319]}} \\
 &  & ACD & $-$0.009 & $-$0.10 & \shortstack{11.83\\{[10.10, 13.53]}} & \shortstack{3.25\\{[0.59, 5.07]}} & \shortstack{$-$0.904\\{[$-$1.011, $-$0.802]}} \\
E-mini & 1 & SE & $-$0.001 & $-$0.18 & \shortstack{0.49\\{[0.21, 0.71]}} & \shortstack{$-$3.24\\{[$-$3.69, $-$2.80]}} & \shortstack{$-$0.347\\{[$-$0.379, $-$0.313]}} \\
 &  & ACD & $-$0.020 & $-$1.49 & \shortstack{6.19\\{[5.29, 7.00]}} & \shortstack{4.69\\{[3.66, 5.84]}} & \shortstack{$-$0.638\\{[$-$0.700, $-$0.585]}} \\
E-mini & 2 & SE & $-$0.033 & $-$0.21 & \shortstack{$-$0.96\\{[$-$1.38, $-$0.60]}} & \shortstack{$-$1.76\\{[$-$2.33, $-$1.30]}} & \shortstack{$-$0.033\\{[$-$0.059, $-$0.001]}} \\
 &  & ACD & 0.548 & 0.49 & \shortstack{0.71\\{[$-$0.13, 1.55]}} & \shortstack{$-$0.88\\{[$-$2.18, 1.00]}} & \shortstack{$-$0.266\\{[$-$0.428, $-$0.107]}} \\
NVDA & 0.5 & SE & $-$0.008 & $-$0.03 & \shortstack{1.38\\{[1.02, 1.64]}} & \shortstack{$-$4.51\\{[$-$5.89, $-$3.37]}} & \shortstack{$-$0.568\\{[$-$0.702, $-$0.460]}} \\
 &  & ACD & $-$0.022 & $-$0.03 & \shortstack{4.56\\{[4.02, 5.08]}} & \shortstack{$-$4.20\\{[$-$6.82, $-$2.11]}} & \shortstack{$-$0.969\\{[$-$1.176, $-$0.812]}} \\
NVDA & 1 & SE & $-$0.003 & 0.14 & \shortstack{0.43\\{[$-$0.01, 0.92]}} & \shortstack{$-$2.09\\{[$-$3.02, $-$1.18]}} & \shortstack{$-$0.244\\{[$-$0.285, $-$0.208]}} \\
 &  & ACD & $-$0.027 & $-$0.23 & \shortstack{5.63\\{[4.89, 6.28]}} & \shortstack{6.93\\{[5.13, 8.70]}} & \shortstack{$-$0.457\\{[$-$0.516, $-$0.379]}} \\
NVDA & 2 & SE & $-$0.049 & $-$0.17 & \shortstack{$-$1.27\\{[$-$1.38, $-$1.13]}} & \shortstack{$-$1.80\\{[$-$2.06, $-$1.51]}} & \shortstack{$-$0.058\\{[$-$0.077, $-$0.039]}} \\
 &  & ACD & 0.269 & 1.14 & \shortstack{1.47\\{[0.89, 2.08]}} & \shortstack{0.29\\{[$-$1.04, 1.47]}} & \shortstack{$-$0.408\\{[$-$0.505, $-$0.304]}} \\
\bottomrule
\end{tabular}
\end{table}

\begin{table}[!tbp]
\centering\small
\caption{Upper-1\% exceedance frequencies. Cells give the mean
exceedance rate (percent) / mean absolute daily deviation from
1\% (percentage points). Dates receive equal weight; the
reference is 1 / 0.}
\begin{tabular}{lrlll}
\toprule
Market & Ticks & MPL & SE & ACD \\
\midrule
E-mini & 0.5 & 1.677 / 0.677 & 2.118 / 1.118 & 2.581 / 1.581 \\
E-mini & 1 & 1.545 / 0.545 & 1.893 / 0.893 & 2.184 / 1.184 \\
E-mini & 2 & 1.030 / 0.280 & 1.072 / 0.313 & 1.528 / 0.546 \\
NVDA & 0.5 & 1.506 / 0.506 & 2.074 / 1.074 & 2.475 / 1.475 \\
NVDA & 1 & 1.753 / 0.753 & 1.996 / 0.996 & 2.209 / 1.209 \\
NVDA & 2 & 1.366 / 0.395 & 1.432 / 0.454 & 1.802 / 0.804 \\
\bottomrule
\end{tabular}
\end{table}

Retaining all available E-mini two-tick dates, which raises that
sample from 29 to 36 dates, makes the intervals for the MPL--ACD
90th- and 99th-percentile loss differences exclude zero: ACD has the
lower 90th-percentile loss and MPL the lower 99th-percentile loss.
The other five samples have no additional dates.
\FloatBarrier

\end{document}